\documentclass{amsart}
\usepackage[dvipsnames]{xcolor}
\usepackage{hyperref}
\hypersetup{
     colorlinks=true,
     linkcolor=blue!60!black,
     filecolor=blue!75!black,
     citecolor = green!50!black,      
     urlcolor=magenta,
     }
\usepackage{stmaryrd}
\usepackage{amssymb}
\usepackage{ bbold }
\usepackage[all]{xy}
\xyoption{all}
\usepackage{url}
\usepackage{mathrsfs}
\usepackage{tikz-cd}
\usepackage{eucal}
\usepackage{comment}
\usepackage{mathtools}
\usepackage{amsthm}
\usepackage{thmtools}
 \usepackage[nameinlink,capitalise,noabbrev]{cleveref}
\usepackage{marginnote}
\usepackage[normalem]{ulem}

\newcommand{\nameop}{\operatorname{op}\nolimits}
\newcommand{\res}{\operatorname{res}\nolimits}
\newcommand{\ind}{\operatorname{ind}\nolimits}
\newcommand{\Hom}{\operatorname{Hom}\nolimits}
\newcommand{\Ext}{\operatorname{Ext}\nolimits}

\def \A{{\mathcal A}}

\def \C{{\mathcal C}}

\def \D{{\mathcal D}}
\def \F{{\mathcal F}}
\def \O{{\mathcal O}}
\def \T{{\mathcal T}}
\def \S{{\mathcal S}}

\def \W{{\mathcal W}}

\def \Fin{{\mathcal{F}\textrm{in}}}

\def \pp{\mathfrak{p}}

\def \Wfin{{\mathcal W}_{G}}

\def \Ofin{{\mathcal O}_{G{\mathcal{F}\textrm{in}}}}
\def \PrL{\mathrm{Pr}^{\mathrm{L}}_{\mathrm{st}}}

\DeclareMathOperator{\Mod}{Mod}

\DeclareMathOperator{\CAlg}{CAlg}

\DeclareMathOperator{\Cof}{Cof}

\numberwithin{equation}{section}
\newtheorem{thmx}{Theorem}

\newtheorem*{introtheorem}{Theorem}

\newtheorem{theorem}[equation]{Theorem}
\newtheorem{proposition}[equation]{Proposition}
\newtheorem{corollary}[equation]{Corollary}
\newtheorem{lemma}[equation]{Lemma}

\theoremstyle{remark}
\newtheorem{definition}[equation]{Definition}

\newtheorem{remark}[equation]{Remark}
\newtheorem{recollection}[equation]{Recollection}

\newtheorem{notation}[equation]{Notation}
\newtheorem{construction}[equation]{Construction}

\newtheorem{warning}[equation]{Warning}


\title[The proper stable module category]
{The proper stable module category of a group algebra}

\begin{document}

\author[]{Georgios Dalezios}
\author[]{Juan Omar G\'omez}

\address{Georgios Dalezios, Dipartimento di Informatica - Settore di Matematica, Università degli Studi di Verona, Strada le Grazie 15 - Ca’ Vignal, I-37134 Verona, Italy}
\email{georgios.dalezios@univr.it}
\urladdr{https://sites.google.com/view/georgiosdalezios/home}

\address{Juan Omar G\'omez, Fakultat f\"ur Mathematik, Universit\"at Bielefeld, D-33501 Bielefeld, Germany}
\email{jgomez@math.uni-bielefeld.de}
\urladdr{https://sites.google.com/cimat.mx/juanomargomez/home}

\begin{abstract}
We introduce the \textit{proper stable module category} for an arbitrary discrete group over any commutative ring by means of cotorsion pairs and abelian model structures. This category is a well-generated tensor-triangulated category and is compactly generated in case the commutative ring is regular. Our construction resembles the topological approach via proper equivariant stable homotopy theory. Moreover, it agrees with the Mazza--Symonds stable module category, whenever the latter is defined, and with various other  stable categories associated to hierarchically defined groups. Along the way, we introduce certain homological dimensions and study them in detail, comparing them with the classical notions.
\end{abstract}

\maketitle

\setcounter{tocdepth}{1}

\tableofcontents

\section{Introduction}

The widespread relevance of the representation theory of groups is evident across numerous fields, even beyond mathematics. This ubiquity is largely due to how deeply the theory is tied to the fundamental understanding of groups. However, representation theory comes in different flavors depending on the choice of both the group and the ground ring, each requiring substantially different techniques. For instance, when working with finite groups over fields whose characteristic does not divide the order of the group, one encounters ordinary representation theory, which is well understood through character theory. By contrast, modular representation theory is generally wild and often demands techniques borrowed from various other areas.

In the latter setting, a key framework that facilitates the transfer of ideas and methods from other contexts is the so-called stable module category, obtained from the category of representations by factoring out the projective modules. This naturally motivates the construction of stable categories in more general situations. One may allow the group to be infinite or, alternatively, keep the group finite while considering coefficients in a more general commutative ring. In both cases, the ad hoc construction by modding out the projectives is generally insufficient; for instance, one would like the resulting stable category not only to be triangulated, but also to carry a compatible monoidal structure. Consequently, one needs constructions that are better adapted to the specific setting.

Many authors have approached this problem by imposing various assumptions on either the group or the coefficient ring in order to identify well-behaved and meaningful subcategories of the category of representations. For example, for finite groups over general commutative rings, the authors of \cite{benson2013module} introduced relative stable categories, which were subsequently studied in \cite{baland2015prime}, \cite{baland2019comparisons} and further refined in \cite{BBIKP25} in the context of commutative Noetherian rings. In particular, the latter work considers the stable category of a certain subcategory of the category of \textit{big} lattices, namely representations that are projective over the ground ring, thereby generalizing the modular setting. This stable category enjoys several desirable properties. For instance, it is a rigidly-compactly generated tt-category. Moreover, many of its tt-geometric aspects are well understood, providing valuable insight into the representation theory in this setting.

When moving beyond finite groups, the picture becomes even less clear. Indeed, arbitrary groups are often too \textit{wild} to admit a satisfactory theory. It is therefore reasonable to focus on classes of groups that behave well with respect to their finite subgroups and to construct stable categories, at least over suitable ground rings; see, for instance, \cite{benson97}, \cite{CEKO14}, \cite{MS19}, \cite{Gom24}. 

Recent literature, which is more related to this paper, demonstrates constructions of stable categories in greater generality, see \cite{ET_cof}, and \cite{emmanouil2025monoidal}. Although the resulting categories remain somewhat elusive - for example, it is not known whether they admit the structure of a tt-category, nor how they behave with respect to restriction to subgroups - these constructions nevertheless constitute significant progress in this direction.

The aim of this article is to address these issues from the following perspective. Namely, we observe that all of the approaches mentioned above seem to seek the same collection of properties, which one might regard as a \textit{wish list} for a stable category of a group over a commutative ring:

\begin{enumerate}
\item It should be a well-generated tensor-triangulated category whose monoidal structure is induced by tensoring over the ground ring, and it should admit an enhancement.
\item The construction should be \textit{functorial} in the group variable with respect to group monomorphisms. In particular, one should have suitable restriction, induction, and coinduction functors.
\item Restriction to the family of all finite subgroups should detect isomorphisms.
\item For finite groups, it should recover the existing stable categories. 
\item It should be \textit{universal} among constructions satisfying the previous properties.
\end{enumerate}

\begin{warning}
  One might argue that another desirable property  for a stable module category is to be a rigidly-compactly generated tensor-triangulated category. However, for infinite groups, even in cases where a well-behaved stable category exists, this is rarely the case; see \cite[Remark 2.7]{gom24b}.
\end{warning}

Note that (d) implicitly assumes that an appropriate notion of stable category for finite groups has already been established. In our setting, we take this to be the stable category of (big) lattices, introduced in \cite{BBIKP25}. As we show in \cref{cor-frobenius} and \cref{prop-monoidal-finite}, this category satisfies properties (a) and (b) and recovers the classical stable module category in the modular setting. We postpone a precise formulation of property (e) until the end of this introduction, since it is most naturally expressed in the language of $\infty$-categories. 
For the moment, we state an informal version of our main result:

\begin{introtheorem}[Informal version]
For every group $G$ and every commutative ring $k$, there exists a tt-category $\mathrm{StMod}^{\mathrm{prop}}_{kG}$ satisfying the properties listed above, which we name the proper stable module category of $G$ over $k$.
\end{introtheorem}

We now explain our results and methods in greater detail, leaving the justification for our terminology to a later point in the introduction.

\subsection*{Methodology} Let $G$ be a group and $k$ a commutative ring. It was recently shown by Emmanouil and Ren \cite{ET_cof} that the class $\mathrm{Cof}(kG)$ of Benson cofibrant modules, introduced in \cite{Ben98}, form the actual cofibrant objects of a combinatorial stable abelian model structure on the category of all $kG$--modules. Our starting point is the observation that, when $G$ is finite, the homotopy category of this model structure coincides with the stable category of $kG$--lattices introduced in \cite{BBIKP25}. 

Motivated by this fact, we seek to extend this construction to arbitrary groups, ensuring that the resulting homotopy category satisfies the requirements of our wish list. Our approach is to work at the level of model categories and to transport the aforementioned model structures coming from \cite{ET_cof} along restriction functors to finite subgroups.

More precisely, we consider the class $\mathcal{W}_G$ of all $kG$--modules $x$ such that, for all finite subgroups $H\leqslant G$, we have $\mathrm{res}^G_H(x)\in\mathrm{Cof}(kH)^{\bot}$. The latter class denotes the right $\mathrm{Ext}^1_{kG}$--orthogonal with respect to $\mathrm{Cof}(kH)$. The class $\mathcal{W}_G$ is our candidate for the trivial (weakly isomorphic to zero) objects, while $\mathrm{PCof}(kG)\coloneqq {}^{\perp}\mathcal{W}_{\mathrm{G}}$ is our candidate for the class of cofibrant objects, that we call \textit{proper cofibrant} $kG$--modules. We point out that all model structures considered in this paper are \textit{abelian} in the sense of Hovey \cite{Hov02} and correspond to certain cotorsion pairs and the so-called \textit{Hovey triples} in the underlying abelian category; the relevant theory is recalled in \cref{sec-prel}.
We are now in a position to state one of our main results, which combines \cref{thm-proper-coto} and \cref{prop-monoidal-proper}.

\begin{thmx}
    The triple $(\mathrm{PCof(kG),\mathrm{PCof}(kG)^{\perp}},\Mod(kG))$ is a hereditary Hovey triple that gives rise to a combinatorial stable abelian monoidal model structure on $\Mod(kG)$ with respect to the monoidal structure given by $\otimes_k$. 
\end{thmx}

The homotopy category of this model structure is by definition the \textit{proper stable module category} $\mathrm{StMod}^{\mathrm{prop}}_{kG}$. It is a triangulated category equivalent to $\mathsf{St}\mathrm{PCof}(kG)$, the stable category of the Frobenius category $\mathrm{PCof}(kG)$.

An important feature of this model structure is that, for every group monomorphism $i\colon H\hookrightarrow G$ the adjoint triple
\begin{center}
          \begin{tikzcd}[column sep=large, row sep=large]
     \Mod(kH) \arrow[r,"\ind_i", yshift=3mm] \arrow[r, yshift=-3mm, "\mathrm{coind}_i"'] & \Mod(kG) \arrow[l, "\res_i" description] 
    \end{tikzcd}
    \end{center}
give rise to Quillen adjuntions $(\ind_i,\res_i)$ and $(\res_i,\mathrm{coind}_i)$; see \cref{lemma-Quillen-adj}. This already constitutes an advantage of our construction, since the analogous statement for the model structure of \cite{ET_cof} is not known to hold in full generality.
Moreover, this observation allows us to strengthen the well-generatedness of the proper stable module category. The following corresponds to \cref{rem:ros} together with \cref{thm-proper-compact}. 

\begin{thmx}
    The proper stable module category $\mathrm{StMod}^{\mathrm{prop}}_{kG}$ is a well-generated tensor-triangulated category which is compactly generated in case $k$ is regular.
\end{thmx}

We further show that our proper stable module category recovers the stable module category of Mazza and Symonds \cite{MS19}, whenever the latter is defined, as well as the stable module category of \cite{ET_cof} for certain hierarchically defined classes of groups; see \cref{subsec-hiera}.

At this stage, our construction addresses points (a)-(d) of the wish list. We now turn to the final point, namely universality. To formulate this property precisely, we work in the setting of stable $\infty$-categories as developed in \cite{Lur17}; we refer the reader to \cref{sec-infty} for further references.  
We emphasize, however, that none of the results in this paper—except those in \cref{sec-infty}—rely on the theory of $\infty$-categories.

We define the \textit{proper stable module $\infty$-category} $\mathbf{StMod}(kG)$ as the underlying $\infty$-category of the proper model structure on $\Mod(kG)$; see \cref{def-inf-proper}. In particular, this construction provides a stable $\infty$-categorical enhancement of $\mathrm{StMod}^{\mathrm{prop}}_{kG}$.
In fact, this construction extends to a functor on the orbit $\infty$-category:
\[
\mathbf{StMod}(k-)\colon \mathcal{O}_G^\mathrm{op} \to \mathrm{CAlg}(\mathrm{Pr}^\mathrm{L}_\mathrm{st}),
\]
see \cref{const-stmod-functor}. Combining this construction with results from \cite{GP26}, we obtain that we can describe the proper stable module $\infty$-category in terms of the ones associated to its finite subgroups. 

The following corresponds to \cref{coro-stmod-descent}. 

\begin{thmx}
     Let $G$ be a group and $k$ a commutative ring. Then the restriction functors induce an equivalence of symmetric monoidal stable $\infty$-categories 
    \[
    \mathbf{StMod}(kG)\xrightarrow[]{\simeq} \lim_{G/F\in \Ofin^{\nameop}}\mathbf{StMod}(kF).
    \]
\end{thmx}

In other words, the above equivalence provides a precise formulation of the universality property in our wish list, since it depends only on properties  (a)-(d). 

\subsection*{Further applications} Since our approach introduces the class $\mathrm{PCof}(kG)$ of proper cofibrant modules, it is natural to ask how this relates to the class $\mathrm{Cof}(kG)$ of Benson cofibrant modules. The following result, corresponding to \cref{rem-Wfin-contains-cof} and \cref{thm:comparison}, provides a more precise comparison between the two notions.

\begin{thmx}
    Let $G$ be a group and $k$ a commutative ring. There is an inclusion $\mathrm{PCof}(kG)\subseteq\mathrm{Cof}(kG)$ which gives rise to a localization sequence of triangulated categories
\begin{equation*}
\xymatrix@C=3pc{
\mathsf{St}\,\mathrm{PCof}(kG) \ar[r]^-{j} & \mathsf{St}\,{\Cof}(kG) \ar[r]^-{\pi} & \mathsf{St}({\mathrm{Cof}(kG)\cap\mathcal{F}\mathrm{Proj}}(kG))
}
\end{equation*}
where $\mathcal{F}\mathrm{Proj}(kG)$  is the class of those $kG$--modules that become  projective after restriction to any finite subgroup of $G$. 
In particular, the canonical inclusion functor $j$ is an equivalence of triangulated categories if and only if $\mathrm{Cof}(kG)\cap\mathcal{F}\mathrm{Proj}(kG)$ agrees with the class of projective $kG$--modules. 
\end{thmx}


We point out that, by a classical result of Cornick and Kropholler \cite{CP}, the class $\mathrm{Cof}(kG)$ is contained in the class $\mathrm{GProj}(kG)$ of Gorenstein projective $kG$--modules. Neither of these classes is known to be closed under tensor products or to be preserved under restriction to arbitrary subgroups of $G$. By contrast, the class $\mathrm{PCof}(kG)$ of proper cofibrants enjoys both of these properties, making it particularly relevant to understanding phenomena in the Gorenstein homological algebra of the group algebra.

Finally, we use the classes of proper cofibrant and Benson cofibrant modules to define the \textit{proper cofibrant cohomological dimension} $\mathrm{PCcd}_kG$ and the \textit{Benson cofibrant cohomological dimension $\mathrm{BCcd}_kG$} of a group $G$ over a commutative ring $k$; see \cref{def-Bensondim}. We study these invariants in detail in \cref{sec-homological}. In particular, the following result combines \cref{thm:BCcd} and \cref{thm:PCcd}. 

\begin{thmx}
    Let $G$ be a group and $k$ a commutative ring with $\mathrm{gldim}(k)<\infty$. Then the following are equivalent:
\begin{itemize}
    \item[(i)] The Benson cofibrant cohomological dimension $\mathrm{BCcd}_{k}G$ \textnormal{(}resp. the proper cofibrant cohomological dimension $\mathrm{PCcd}_kG$\textnormal{)} is finite.
    \item[(ii)] There exists a $k$-split short exact sequence $0\rightarrow k\rightarrow w\rightarrow c\rightarrow 0$ of $kG$--modules where $\mathrm{pd}_{kG}(w)<\infty$ and $c\in\mathrm{Cof}(kG)$ \textnormal{(}resp. $c\in \mathrm{PCof}(kG)$\textnormal{)}.
    \item[(iii)] All $kG$--modules have finite Benson cofibrant dimension \textnormal{(}resp. proper cofibrant dimension\textnormal{)}.
    \item[(iv)] $\mathrm{spli}(kG)<\infty$ and $\mathrm{Cof}(kG)=\mathrm{GProj(kG)}$ \textnormal{(}resp. $\mathrm{PCof}(kG)=\mathrm{Cof}(kG)=\mathrm{GProj(kG)}$\textnormal{)}. 
\end{itemize}
\end{thmx}

\subsection*{A topological approach}
We briefly explain the terminology ``proper'', which is motivated by the corresponding topological notion.

The work of the second-named author with Luca Pol \cite{GP26b} is motivated by the broader goal of understanding representation-theoretic aspects of groups through the lens of proper equivariant homotopy theory; see \cite{degrijse2023proper}. More precisely, in \textit{loc. cit.}, certain categories of group representations are described in terms of module categories internal to the category $\mathrm{Sp}_G$ of proper $G$--spectra, over suitable equivariant ring spectra.

We will not discuss the details of this construction here. Rather, the following result should be viewed 
simply as motivation for the present work.

\begin{introtheorem}
    [G\'omez--Pol]
Let $k$ be a Noetherian ring of finite global dimension, and let $G$ be a group of type $\Phi_k$. Then there is an equivalence of symmetric monoidal stable $\infty$-categories
\[
\mathbf{StMod}(kG) \simeq \mathrm{Mod}_{\mathrm{Sp}G}(b_G(k)\otimes \widetilde{EG}),
\]
where $b_G(k)$ denotes the equivariant Eilenberg--MacLane spectrum associated to the Borel construction of $k$, and $\widetilde{EG}$ denotes the cofiber of the canonical map $EG_+ \rightarrow S^0$.
\end{introtheorem}


The crucial observation is that the
above equivalence appears to rely only on properties analogous to those appearing in our wish list. This leads us to conjecture that such an equivalence should hold for every group, provided that the ground ring is regular.


\subsection*{Acknowledgments} We would like to thank Luca Pol for his insights on \cref{sec-functoriality}.    GD acknowledges support by the project LAVIE – Large views of small phenomena:
decompositions, localizations, and representation type, FIS 00001706, funded by Program FIS2021 of Italian Ministry
of University and Research. JOG was supported by the Deutsche Forschungsgemeinschaft (Project-ID 491392403 – TRR 358).

\subsection*{Notation and conventions} 
Throughout this paper, $G$ denotes an arbitrary group and $k$ an arbitrary commutative ring with unit, unless additional assumptions are explicitly stated. We write $\Mod(kG)$ for the category of all (left) $kG$--modules. For a  Frobenius exact category  $\mathcal{E}$, we write $\mathsf{St}\mathcal{E}$ to denote its stable category.

\section{Preliminaries} \label{sec-prel}

This section is devoted to establishing some notation and terminology regarding cotorsion pairs on abelian categories, model structures, modules over group algebras and Benson cofibrant modules. 

\subsection{Abelian model categories} We assume familiarity with the language of model categories and refer to \cite{Hov99} and \cite{Hov02} for unexplained terminology. For convenience, we briefly recall some relevant results; see also  \cite{Bec14}.

\begin{notation}
    Let $\A$ be an abelian category and $\mathcal{X}$ be a 
    class of objects in $\A$. We write $\mathcal{X}^\perp$ to denote the class of all objects $y\in \A$ satisfying $\mathrm{Ext}^1_\A(x,y)=0$ for all $x\in \mathcal{X}$. 
    Here,  $\mathrm{Ext}^1_\A(-,-)$ denotes the Yoneda Ext bifunctor. Similarly, we write ${}^\perp\mathcal{X}$ to denote the class of all objects $x\in \A$ satisfying $\mathrm{Ext}^1_\A(x,y)=0$ for all $y\in \mathcal{X}$.
\end{notation}

\begin{definition}
    Let $\A$ be an abelian category. Consider a  pair $(\C,\F)$ of classes of objects of $\A$. We say that $(\C,\F)$ is a \textit{cotorsion pair} if $\C={}^\perp\F$ and $\C^\perp=\F$. Moreover, we say that a cotorsion pair $(\C,\F)$ is:
    \begin{itemize}
        \item \textit{small} if there is a set of objects $\S$ such that $\S^\perp=\F$ (in this case, we also say that $(\C,\F)$ is generated by the set $\S$);
        \item \textit{complete} if for every object $A\in\A$ there exists a short exact sequence $0\rightarrow F\rightarrow C\rightarrow A\rightarrow 0$ with $C\in\C$ and $F\in\F$, and also a short exact sequence $0\rightarrow A\rightarrow F'\rightarrow C'\rightarrow 0$ with $C'\in\C,$ and $F'\in\F$;
        \item \textit{hereditary} if for all $i\geqslant 1,\,  C\in C$ and $F\in\F$ we have $\mathrm{Ext}_{\A}^i(C,F)=0.$
    \end{itemize}
\end{definition}

\begin{recollection}
\label{rem:complete}
If $\A$ is a Grothendieck category with enough projectives then any cotorsion pair $(\mathcal{C,\mathcal{F}})$ in $\mathcal{A}$ which is generated by a set $\mathcal{S}$ is complete, see \cite[Corollary~6.6]{Hov02}. The original argument goes back to \cite[Theorem~10]{Eklof} (the terminology used in these sources differs than ours - we stick with using ``generated'' rather than ``cogenerated'' for a small cotorsion pair, which is by now customary). In this case, it is well known that $\mathcal{C}={}^{\oplus}\mathrm{Filt}(\mathcal{S})$, that is, $\C$ consists of summands of transfnite extensions of objects from $\mathcal{S}$, see for instance \cite[Corollary~6.9]{Hov02}. 
In the context of Hovey's correspondence, which is recalled below as  \cref{thm:Hovey}, this can be viewed as a version of the small object argument, see \cite{Jan-Man}.
\end{recollection}

\begin{recollection}
    Let $\A$ be a complete and cocomplete abelian category that admits a model structure. The objects of $\mathcal{A}$ that are weakly equivalent to the zero object are called \textit{trivial}. Following \cite{Hov02} we say that the model structure on $\A$ is \textit{abelian} if the (trivial) cofibrations are precisely the monomorphisms with (trivially) cofibrant cokernel and the (trivial) fibrations are precisely the  epimorphisms with (trivially) fibrant kernel. 
    
    Now, assume that we have a triple $(\C,\W,\F)$ of full subcategories of $\A$. We let $\mathrm{Cofib}$ denote the class of monomorphisms with cokernel in $\C$; $\mathrm{Fib}$ the class of epimorphisms with kernel in $\mathrm{Fib}$; and $\mathrm{Weak}$ the class of maps that factor as $p\circ i$ where $i$ is a monomorphism with cokernel in $\C\cap\W$ and $p$ is an epimorphism with kernel in $\W\cap\F$. 
    One says that the triple $(\C,\W,\F)$ determines an abelian model structure on $\A$ if the triple $(\mathrm{Cofib},\mathrm{Weak},\mathrm{Fib})$ is an abelian model structure on $\A$.
\end{recollection}

There is a close relation between cotorsion pairs and abelian model structures which goes back to \cite[Theorem 2.2]{Hov02}.

\begin{theorem}[Hovey]
\label{thm:Hovey}
    Let $\A$ be a complete and cocomplete abelian category and $(\C,\W,\F)$ a triple of classes of objects. Then the following are equivalent: 
    \begin{enumerate}
        \item  $\A$ admits an abelian model structure, with $\C, \F$ and $\W$ being the classes of cofibrant, fibrant and trivial objects repectively.
        \item Both $(\C,\W\cap \F)$ and $(\C\cap \W,\F)$ are complete cotorsion pairs, and $\W$ is thick; that is, $\W$ closed under direct summands and satisfies the 2-out-of-3 property for short exact sequences in $\A$.
    \end{enumerate}
    Moreover, if additionally  $(\C,\W\cap \F)$ and $(\C\cap \W,\F)$ are small cotorsion pairs, then the resulting model structure on $\A$ is cofibrantly generated. 
\end{theorem}

\begin{definition}
\label{def:hovey_triple}
In the context of Theorem~\ref{thm:Hovey} we say that $(\C,\W,\F)$ is a \textit{Hovey triple} on the abelian category $\A$.
\end{definition}


\begin{remark}
The converse of the last sentence of Theorem \ref{thm:Hovey} is also valid if $\A$ is Grothendieck (\cite[Proposition~1.2.7]{Bec14}), or even a locally presentable abelian category (\cite[Lemma~3.7]{PoSt22}). Hence in these cases combinatorial abelian model structures correspond to Hovey triples whose cotorsion pairs are small.
\end{remark}

\begin{definition}
    A Hovey triple $(\C,\W,\F)$ on $\A$ is called \textit{hereditary} if both cotorsion pairs $(\C,\W\cap \F)$ and $(\C\cap \W,\F)$ are hereditary. This is equivalent to asking that $\C$ is closed under kernels of epimorphisms  in $\A$, or that $\F$ is closed under cokernels of monomorphisms in $\A$; see for instance \cite[Corollary~1.1.12]{Bec14}.
\end{definition}

Recall that an extension-closed subcategory of an abelian category is called \textit{Frobenius}, if it has enough (relative) projectives and enough (relative) injectives such that these classes actually coincide.

Let us record the following result which corresponds to \cite[Proposition 5.2]{Gil11} { and }\cite[Theorem~2.6]{Gil_survey}.

\begin{proposition}
\label{prop:her_cat}
{Let $(\C,\W,\F)$ be a hereditary Hovey triple on a complete and cocomplete abelian category $\A$.} 
Then the full subcategory $\C\cap \F$ is a Frobenius exact category with the exact structure inherited from $\A$. The  projective-injective objects of $\C\cap\F$ are given by $\C\cap\W\cap\F$. { In addition}, the stable category of the Frobenius category $\C\cap\F$, which is triangulated by \cite{Happel}, is equivalent to the homotopy category of the corresponding abelian model structure on $\A$. 
\end{proposition}

\subsection{Generalities on group algebras} 
Our main reference for this subsection is \cite{benson97}.

\begin{recollection}\label{rec-adjuntions-mod}
The category $\Mod(kG)$ is a closed symmetric monoidcal category where the monoidal product is given by the tensor product $\otimes_k$ of $kG$--modules with diagonal $G$-action and the monoidal unit is $k$ with the trivial action of $G$. In particular, for any $kG$--module $x$ the endofunctor $x\otimes_k-$ has a right adjoint $\hom_k(x,-)$. In case the ring of coefficients is clear from the context, we just denote these functors by $x\otimes-$ and $\hom(x,-)$. 

 For any group homomorphism  $i\colon H\to G$, restricting the action of a $kG$--module along the map $i$ gives rise to the so-called restriction along $i$ functor, which is denoted by $\res_i$. Moreover, this functor has a left and a right adjoint 
\begin{center}
          \begin{tikzcd}[column sep=large, row sep=large]
     \Mod(kH) \arrow[r,"\ind_i", yshift=3mm] \arrow[r, yshift=-3mm, "\mathrm{coind}_i"'] & \Mod(kG) \arrow[l, "\res_i" description] 
    \end{tikzcd}
    \end{center}
and one refers to $\ind_i$ and $\mathrm{coind}_i$ as the induction along $i$ functor and coinduction along $i$ functor, respectively. Note that this is just a particular case of (co)extension of scalars in ring theory for the map of rings $kH\to kG$ induced by $i$. 
\end{recollection}

\begin{notation}
    In case $H\leqslant G$ is a subgroup, we write  $\mathrm{res}^G_H(-)$ or $- \downarrow_H^G$ to denote the restriction functor, and similarly for the induction and coinduction functors. In particular,  $\ind_H^G$ is given by  $RG\otimes_{kH}-$  (also denoted by $- \uparrow^G_H$) and the coinduction functor $\mathrm{coind}_H^G$ corresponds to   $\Hom_{kH}(kG,-)$ (also denoted by $- \Uparrow^G_H$).
\end{notation}

\begin{recollection}
 Let $H$ be a subgroup of $G$. Given a $kG$--module $x$ and a  $kH$--module $y$, we have the following isomorphisms of $kG$--modules: 
 \[
 x\otimes \ind_H^G y\simeq \ind_H^G(\res^G_H (x)\otimes y).
 \]
 This is known as the  \textit{projection formula}; see \cite[Proposition 3.3.3]{Ben98}.
\end{recollection}

\begin{recollection}
  Let  $H$ and $K$ be subgroups of $G$. One can describe what happens to a module induced from $K$ and then restricted to $H$ by means of the \textit{Mackey formula} (also known as the \textit{double coset formula}): Let $x$ be a $kK$--module. Then we have an isomorphisms of $kH$--modules
  \[
   \res^G_H\ind_K^Gx\simeq \bigoplus_{g\in H\backslash G/K} \ind_{{}^g\!K\cap H}^H \res^{{}^g\!K}_{{}^g\!K\cap H}({}^g\!x)
  \]
  where ${}^g\!K=gKg^{-1}$, and ${}^g\!x$ denotes the $k({}^g\!K)$--module obtained from $x$ via  restriction along the conjugation map $c^\ast_g$. 
\end{recollection}

\subsection{Modules of bounded functions}
Let $B(G,k)$ denote the set of bounded functions from $G$ to $k$, that it, the functions which take only finitely many distinct values. Notice that $B(G,k)$ is a $kG$--module by the rule $gf(x)=f(g^{-1}x)$. In addition, it is a commutative ring with pointwise addition and multiplication, in fact, the ring homomorphism $\iota\colon k\rightarrow B(G,k)$ which sends an element $r\in k$ to the constant function at $k$ makes $B(G,k)$ a commutative $k$--algebra. For $k=\mathbb{Z}$ the map $\iota$ is $\mathbb{Z}G$--linear, in fact, the short exact sequence of $\mathbb{Z}G$--modules
\begin{equation}
0\rightarrow \mathbb{Z}\xrightarrow{\iota} B(G,\mathbb{Z}) \xrightarrow{\pi} C\rightarrow 0 \nonumber
\end{equation}
splits over $\mathbb{Z}$; the map $s\colon B(G,\mathbb{Z})\rightarrow \mathbb{Z};\, f\mapsto f(1)$ provides a $\mathbb{Z}$--splitting of $\iota$. The induced short exact sequence of (diagonal) $kG$--modules
\begin{equation}
\label{eq:B_sequence}
0\rightarrow \mathbb{Z}\otimes_{\mathbb{Z}}k \xrightarrow{\iota\otimes k} B(G,\mathbb{Z})\otimes_{\mathbb{Z}}k \xrightarrow{\pi\otimes k} C\otimes_{\mathbb{Z}}k\rightarrow 0
\end{equation}
splits over $k$; the map $s\otimes k$ provides a $k$--splitting of $\iota\otimes k$.
Notice that there is an isomorphism $B(G,\mathbb{Z})\otimes_{\mathbb{Z}}k\cong B(G,k)$ of $kG$--modules.
In addition, $B(G,k)$ is a free $k$--module, and for any finite subgroup $H\leqslant  G$ the $RH$--module $\res_H^GB(G,k)$ is free, see \cite[Lemma~3.4]{benson97}. Sometimes we just write $B$ instead of $B(G,k)$. 

\medskip
The following concept was introduced in \cite[Definition~4.1]{benson97}.

\begin{definition}
\label{def:cofibrants}
A $kG$--module $x$ is called \textit{cofibrant} if $B(G,k)\otimes_kx$ is a projective $kG$--module. We will denote by $\mathsf{Cof}(kG)$ the class of cofibrant $kG$--modules and by $\mathsf{cof}(kG)$ the class of cofibrant $kG$--modules which are finitely generated as $k$--modules.
\end{definition}

Let us recall the following result which corresponds to \cite[Proposition 4.12]{ET_cof}.

\begin{proposition}\label{prop-model-struc}
The triple  $(\mathrm{Cof}(kG),\mathrm{Cof}(kG)^\perp,\Mod(kG))$ {is a hereditary Hovey triple} which determines a combinatorial abelian model structure on $\Mod(kG)$ whose homotopy category is {equivalent to} the stable category of the Frobenius exact category $\mathrm{Cof}(kG)$, which is denoted by $\mathsf{St}\mathrm{Cof}(kG)$.  
\end{proposition}

\begin{remark}
\label{rem:properties of cofs}
    From the previous result it follows that the class $\mathrm{Cof}(kG)$ is closed under coproducts, extensions, kernels of epimorphisms and contains all the projective $kG$--modules.   
\end{remark}

We aim to illustrate how cofibrant objects behave with respect to induction and restriction functors. We need some preparation. 

\begin{proposition}\label{prop-tensor-proje-latti}
    Let $G$ be a group, $x$ be a $kG$--module which is projective as $k$--module and $y$ be a projective $kG$--module. Then the module $x\otimes y$ is projective over $kG$.  
\end{proposition}

\begin{proof}
    Consider the natural isomorphism 
\[
\Hom_{kG}(x\otimes y,-) \cong \Hom_{kG}(y,\hom(x,-)).
\]
Since $y$ is projective as a $kG$--module and $x$ is projective as a $k$--module, the statement follows.
\end{proof}

The next observation is pivotal for this work. It is already mentioned in \cite{BG}, but we include a proof for convenience. 

\begin{proposition}
\label{prop:cof_proj}
The cofibrant $kG$--modules are projective over the coefficient ring $k$. The converse holds in case $G$ is finite. 
\end{proposition}

\begin{proof}
Let $x$ be a cofibrant $kG$--module. We may apply the functor $-\otimes x$ to the short exact sequence \eqref{eq:B_sequence} and deduce the following $k$-split short exact sequence of diagonal $kG$--modules,
\begin{equation}
\label{eq:B_sequence2}
0\rightarrow k\otimes x \rightarrow B\otimes x \rightarrow (C\otimes_{\mathbb{Z}}k)\otimes x\rightarrow 0. \nonumber
\end{equation}
Since $B\otimes_kx$ is a projective $kG$--module it is also projective as $k$--module. 
Thus, $x$ is projective as $k$--module. Now assume that $G$ is finite and that $x$ is a $kG$--module which is projective over $k$. In this case, $B$ is a projective $kG$--module and hence $B\otimes x$ is projective as well by \cref{prop-tensor-proje-latti}.  It follows that $x$ is cofibrant.
\end{proof}

\begin{remark}
\label{rem:cof_fg}
It follows by  \cref{prop:cof_proj} that the modules in $\mathsf{cof}(kG)$ as in  \cref{def:cofibrants} are finitely generated as $kG$--modules. 
\end{remark}

\begin{proposition}
\label{prop:cof_ind_kes}
Consider a finite subgroup $H\leqslant G$. Let $y$ be a cofibrant $kG$--module and $x$ a cofibrant $kH$--module. Then the following hold:
\begin{itemize}
\item[(i)] The module $\res^G_H y$ is $kH$--cofibrant. 
\item[(ii)] The $kH$--module $x\otimes \res^G_HB$ is projective.
\item[(iii)] The induced $kG$--module $\ind^G_H x$ is  cofibrant.
\end{itemize}
\end{proposition}

\begin{proof}
(i) By  \cref{prop:cof_proj}, we know that $y$ is $k$--projective, hence so is $\res^G_Hy$. Another layer of  \cref{prop:cof_proj} gives us that $\res^G_H y$ must be cofibrant. 

(ii) This follows by  \cref{prop-tensor-proje-latti} since $\res^G_HB$ is a free $kH$--module. 

(iii) We need to prove that the $kG$--module 
\[
(kG\otimes_{kH}x) \otimes B \cong kG\otimes_{kH} (x\otimes \res^G_H B )
\]
is projective. The $kH$--module $x\otimes \res^G_H B$ is projective by (ii), thus its induction from $H$ to $G$ results in a projective $kG$--module.
\end{proof}

\begin{remark}
    Note that when $x$ is finitely generated over $k$ then so is $x\otimes B(H,k)$, but $kG\otimes_{kH}x$ is not necessarily finitely generated over $k$; it depends on the index of $H$ in $G$.
\end{remark}

\subsection{Finite groups}
In \cite[Section~2]{BBIKP25} for a $k$--algebra $A$ which is finitely generated and projective as a $k$--module, the authors denote by $\mathrm{Mod}(A,k)$ the category which consists of the $A$--modules that are projective as $k$--modules and by $\mathrm{mod}(A,k)$ the full subcategory of $\mathrm{Mod}(A,k)$ consisting of the $A$--modules that are finitely generated as $k$--modules. Notice that the modules in $\mathrm{mod}(A,k)$ will be finitely generated over $A$ too. The following is a consequence of the definitions and \cref{prop:cof_proj}.

\begin{proposition}
\label{prop:identifications}
For a finite group $G$ there are identifications $\mathrm{Mod}(kG,k)=\mathrm{Cof}(kG)$ and $\mathrm{mod}(kG,k)=\mathrm{cof}(kG)$.
\end{proposition}

\begin{remark}
    We point out that the modules of $\mathrm{mod}(kG,k)$ are also called $kG$ \textit{lattices}. Sometimes we refer to the modules of $\mathrm{Mod}(kG,k)$ as big lattices.
\end{remark}

Combining the above observation with \cref{prop-model-struc}, we get the following result:

\begin{corollary}\label{cor-frobenius}
For a finite group $G$ the group algebra $kG$ is a Frobenius $k$--algebra in the language of \cite{BBIKP25}; in particular, both categories $\Mod(kG,k)$ and $\mathrm{mod}(kG,k)$ are Frobenius exact with pro-injective objects given by the projective $kG$--modules.  
\end{corollary}

\begin{proof}
    By \cite[Lemma~2.6]{BBIKP25}, it is enough to verify that $\Mod(kG,k)$ is Frobenius exact. By \cref{prop-model-struc} we know that $\mathrm{Cof}(kG)$ is Frobenius exact, hence the conclusion follows by \cref{prop:identifications}.
\end{proof}


As in \cite{BBIKP25} by a localizing subcategory $\mathcal{L}$ of an exact additive category $\mathcal{A}$, we mean a full subcategory which satisfies the 2-out-of-3 property for conflations and is closed under coproducts. The intersection of localizing subcategories is a localizing subcategory, hence it makes sense to consider the smallest localizing subcategory of $\mathcal{A}$ generated by a class $\mathcal{X}$ of objects of $\mathcal{A}$, which we denote by $\mathrm{Loc}(\mathcal{X})$.

\begin{recollection}\label{rec-modlf-comp}
By \cite[2.5]{BBIKP25} for a finite group $G$ the smallest localizing subcategory of $\mathrm{Mod}(kG,k)$ containing $\mathrm{mod}(kG,k)$ is denoted by $\mathrm{Mod}^{\mathrm{lf}}(kG,k)$. It is a Frobenius exact category, with the associated stable category denoted by $\mathsf{St}\mathrm{Mod}^{\mathrm{lf}}(kG,k)$. In addition, by \cite[Proposition~2.9]{BBIKP25} this triangulated category is compactly generated and its subcategory of compact objects identifies with the idempotent completion of $\mathsf{St}\mathrm{mod}^{\mathrm{lf}}(kG,k)$. In case $k$ is regular, that is, $k$ is Noetherian and $k_\pp$ has finite global dimension for all primes $\pp$ of $k$, there is an identification $\mathrm{Mod}^{\mathrm{lf}}(kG,k)=\mathrm{Mod}(kG,k)$ by \cite[Proposition~2.14]{BBIKP25}.
\end{recollection}  

A relevant consequence of the previous recollection is the following corollary that we will use later.  

\begin{corollary}\label{rec-regular-ring}
Let $G$ be a finite group and $k$ a regular commutative ring. Then $\mathsf{St}\mathrm{Cof}(kG)$ is a compactly generated triangulated category.
\end{corollary}

\begin{proof}

Follows from \cref{rec-modlf-comp} together with  \cref{prop:identifications}.
\end{proof}


\section{The proper model structure}
\label{sec-proper}

Let $G$ be  a group and $k$ be a commutative ring. We write $\Fin(G)$ to denote the family of all finite subgroups of $G$. Recall from \cref{prop-model-struc} the Hovey triple $(\Cof(kG),\Cof(kG)^\perp,\Mod(kG))$ and its corresponding abelian model structure. 

Our goal is to transfer this abelian model structure along the family of adjunctions 
\[\{\res^G_F\colon \Mod(kG) \rightleftarrows \Mod(kF) \colon\ind^G_F\}_{F\in\Fin(G)}.\]
This is achieved with \cref{thm-proper-coto} below. 
For a general discussion on transferring Hovey triples along a single adjunction the reader may consult \cite[Section~5]{Dal-Psa}. Here we give a quick and independent proof of  \cref{thm-proper-coto} without referring to results from  \cite{Dal-Psa}.

\begin{definition}\label{def-prop-model}
    We let $\Wfin$ denote the class of $kG$--modules $x$ such that $\res^G_Fx\in \Cof(kF)^\perp$ for all $F\in \Fin(G)$. 
    We write $\mathrm{PCof}(kG)$ to denote ${}^\perp\Wfin$ and refer to it as the class of \textit{proper cofibrant objects}.  
\end{definition}

\begin{proposition} 
\label{rem:finite_W}
If $G$ is a finite group then $\Wfin=\Cof(kG)^\perp$; in particular, $\mathrm{PCof}(kG)=\mathrm{Cof}(kG)$. In addition, for any group $\Gamma$ one has $x\in\W_{\Gamma}$ if and only if $\res^{\Gamma}_F(x)\in\mathcal{W}_{F}$ for all finite subgroups $F\leqslant  \Gamma$.
\end{proposition}

\begin{proof}
The containment $\Wfin\subseteq\Cof(kG)^\perp$ is clear since $G$ is finite. For the converse, let $F\leqslant G$ be a subgroup, $y\in\Cof(kG)^\perp$ and $x\in\Cof(kF)$. Then    \[
\mathrm{Ext}^1_{kF}(x,\res^G_F y)\cong  \mathrm{Ext}^1_{kG}(\ind^G_Fx,y)=0;
\] 
where we used the Eckmann-Shapiro Lemma (\cite[Corollary~2.8.4]{Ben98}) and \cref{prop:cof_ind_kes} (iii). Hence $y\in\Wfin$ which concludes the proof.
\end{proof}

The relevance of  \cref{def-prop-model} lies in the following result.

\begin{theorem}\label{thm-proper-coto}
    The triple $(\mathrm{PCof}(kG),\mathrm{PCof}(kG)^\perp,\Mod(kG))$ is a hereditary Hovey triple which gives rise to a combinatorial stable abelian model structure on $\Mod(kG)$. 
\end{theorem}   

\begin{remark}
It will follow from \cref{thm-proper-coto} that $\mathrm{PCof}(kG)^\perp = \Wfin$, since it will be proved that $(\mathrm{PCof}(kG), \Wfin)$ is a cotorsion pair. However, after the proof of \cref{thm-proper-coto} we shall use these two notations interchangeably whenever no ambiguity can arise.
\end{remark}

 Let us spell out some direct consequences of the previous theorem, which follow from the recollections of \cref{prop:her_cat}.

\begin{corollary}
     In the context of the previous theorem, the category of fibrant-cofibrant objects $\mathrm{PCof}(kG)$ is a Frobenius exact category with the canonical exact structure induced from $\Mod(kG)$, and the homotopy category of  $\Mod(kG)$ is a triangulated category equivalent to  the stable category of this Frobenius category. 
\end{corollary}

\begin{definition} 
\label{def:proper}
    We write $\mathrm{StMod}^\mathrm{prop}_{kG}$ to denote the homotopy category of $\Mod(kG)$ (equipped with the above model structure) and refer to it as the \textit{proper stable module category of $kG$}.  
\end{definition}

\begin{remark}
\label{rem:ros}
The proper stable module category is well-generated triangulated category in the sense of Neeman; see \cite[Proposition 6.10]{Ros05}.    
\end{remark}

We need some preparations towards proving \cref{thm-proper-coto}. 

\begin{notation} 
By  \cref{prop-model-struc}, there exists a set $\mathcal{I}_G$ of objects of $\Mod(kG)$ that  generates the cotorsion pair $(\Cof(kG),\Cof(kG)^\perp)$, that is, $\Cof(kG)^\perp = \mathcal{I}_G^\perp$.
\end{notation}

\begin{proposition}\label{prop-equality-wfin}
    There is an equality 
    \[
    \Wfin=\{y\in \Mod(kG)\mid \forall F\in \Fin(G), \, \forall x\in \mathcal{I}_F : \mathrm{Ext}^1_{kG}(\ind_F^Gx,y)=0\}. 
    \]
\end{proposition}

\begin{proof}
    The equality is a direct consequence of the Eckmann-Shapiro Lemma (see for instance \cite[Corollary 2.8.4]{Ben98}): 
    \begin{align*}
        y\in \Wfin & \iff \forall F\in \Fin(G), \, \res^G_Fy\in \Cof(kF)^\perp \\
        & \iff \forall F\in \Fin(G), \forall x\in \Cof(kF) : \mathrm{Ext}^1_{kF}(x,\res^G_Fy)=0 \\ 
        & \iff  \forall F\in \Fin(G), \forall x\in \mathcal{I}_F : \mathrm{Ext}^1_{kF}(x,\res^G_Fy)=0\\
        & \iff  \forall F\in \Fin(G), \forall x\in \mathcal{I}_F : \mathrm{Ext}^1_{kG}(\ind_F^G x,y)=0;
    \end{align*}
hence the result follows. 
\end{proof}

\begin{corollary}\label{coro-wfin-complete}
    The pair $({}^\perp\Wfin,\Wfin)$ is a small cotorsion pair of $\Mod(kG)$. In particular, it is complete.  
\end{corollary}

\begin{proof}
    By \cref{prop-equality-wfin}, we obtain that $({}^\perp\Wfin,\Wfin)$ is the cotorsion pair {generated by} the set 
    \[
    \mathcal{S}_G\coloneqq \{\ind_F^G x\mid x\in \mathcal{I}_F,\, F\in \Fin(G)\}. 
    \]
    Thus the result follows by  \cref{rem:complete}.
\end{proof}

\begin{proposition}\label{prop-Wfin-hereditary}
    The class $\Wfin$ satisfies the 2-out-of-3 property. In particular, $({}^\perp\Wfin,\Wfin)$ is a hereditary cotorsion pair. 
\end{proposition}

\begin{proof}
Let $0\to x\to y\to z \to 0$ be a short exact sequence  with two out of the three terms  in $ \Wfin$. We claim that the third term is in $\Wfin$ too. Indeed, for any $F\in \Fin(G)$, we obtain a short exact sequence 
\[
0 \to \res^G_F x\to \res^G_F y\to \res^G_F z\to 0
\]
since the restriction functor is exact. Now, the two terms in the former short exact sequence satisfy that their image under the restriction functor lies in $\Cof(kF)^\perp$,  and hence so does the image of the third term since $\Cof(kF)^\perp$ is thick; see \cref{prop-model-struc}. The last claim follows since a complete cotorsion pair is coresolving if and only if it is resolving; see for instance \cite[Corollary 1.1.12]{Bec14}.
\end{proof}

\begin{proof}[Proof of \cref{thm-proper-coto}]
   {In view of \cref{thm:Hovey} and  \cref{prop:her_cat} we only need to} verify that $({}^\perp\Wfin,\Wfin)$ and $({}^\perp\Wfin\cap\Wfin,\Mod(kG))$ are small and hereditary cotorsion pairs in $\Mod(kG)$, with the class $\Wfin$ being thick. Notice that the first cotorsion pair is settled by the results proved above. 
   
   Let us now prove that the class $\Wfin$ is thick. Indeed, it is closed under direct summands  since the restriction functor is additive and $\Cof(kF)^\perp$ is closed under direct summands for any group $F$. Moreover, it satisfies the 2-out-of-3 property by \cref{prop-Wfin-hereditary}. It remains to verify that ${}^\perp\Wfin\cap\Wfin$ agrees with $\mathrm{Proj}(kG)$, the class of projective $kG$--modules. Any projective $kG$--module module belongs in ${}^\perp\Wfin$. In addition, since restriction preserves projective modules, and 
   \[
   \Cof(kF)\cap \Cof(kF)^\perp=\mathrm{Proj}(kF),
   \]
    it follows that   $\mathrm{Proj}(kG)\subseteq {}^\perp\Wfin\cap\Wfin$. 
    For the reverse inclusion, let $x$ be a $kG$--module in ${}^\perp\Wfin\cap\Wfin$. Then there is a projective $kG$--module $p$ and a short exact sequence 
    \[
    0\to \ y\to p \to x\to 0.
    \]
    By the 2-out-of-3 property we deduce that $y$ is in $\Wfin$, and hence the short exact sequence splits. It follows that $x$ is projective.
\end{proof}

The following result gives the relation between the class $\mathrm{PCof}(kG)$ of proper cofibrants and the class $\mathrm{Cof}(kG)$ of Benson cofibrants.

\begin{proposition}\label{rem-Wfin-contains-cof}
The inclusion $\Cof(kG)^\perp\subseteq\Wfin$ holds true; equivalently, we have that $\mathrm{PCof}(kG)\subseteq \Cof(kG)$. In particular, the modules in $\mathrm{PCof}(kG)$ are $k$--projective. 
\end{proposition}

\begin{proof}
Fix  $y\in \Cof(kG)^\perp$. By \cref{prop:cof_ind_kes},  for any $F\in \Fin(G)$, we have that  $\ind_F^G\mathcal{I}_F\subseteq \Cof(kG)$. In particular, for all $x\in \mathcal{I}_F$, we obtain that 
\[
0=\mathrm{Ext}^1_{kG}(\ind^G_Fx,y)\simeq \mathrm{Ext}^1_{kF}(x,\res^G_F y)
\] 
and hence $\res_F^Gy\in \Cof(kF)^\perp$. Since $F$ was an arbitrary finite subgroup of $G$, we deduce that $y\in \Wfin$, which proves that $\Cof(kG)^\perp\subseteq\Wfin$. This is equivalent to $\mathrm{PCof}(kG)\subseteq \Cof(kG)$ since both $\mathrm{PCof}(kG)$ and $\mathrm{Cof(kG)}$ form the left hand sides of cotorsion pairs. The last claim follows from \cref{prop:cof_proj}.
\end{proof}

We now prove the existence of certain Quillen functors between proper stable module categories, which will be useful in the sequel.



\begin{lemma}\label{lemma-Quillen-adj}
    Let $i\colon H\to G$ be a group monomorphism. Then in the adjoint triple  
    \begin{center}
          \begin{tikzcd}[column sep=large, row sep=large]
     \Mod(kH) \arrow[r,"\ind_i", yshift=3mm] \arrow[r, yshift=-3mm, "\mathrm{coind}_i"'] & \Mod(kG) \arrow[l, "\res_i" description] 
    \end{tikzcd}
    \end{center}
the pairs $(\mathrm{ind_i,\res_i})$ and  $(\res_i,\mathrm{coind}_i)$ are Quillen adjunctions with respect to the model structures induced by \cref{thm-proper-coto}. 
\end{lemma}

{
\begin{proof}
Recall the existence of the adjuctions at the level of abelian categories (\cref{rec-adjuntions-mod}) and notice that all the functors involved are exact. Indeed, restriction is clearly exact and induction and coinduction are exact since $\res_i kG$ is projective as a $kH$--module.  For simplicity, we assume that $H\leqslant G$, the general case follows in a similar fashion. 

In order to prove that $(\mathrm{ind}_H^G,\res_H^G)$ is a Quillen adjunction, it suffices to prove that $\res_H^G$ is right Quillen, i.e. that it maps (trivial) fibrations to (trivial) fibrations. Since $\res_H^G$ is exact, this amounts to checking that it maps (trivially) fibrant objects to (trivially) fibrant objects. Since all objects are fibrant in the involved model structures we just need to prove that $\res_H^G$ maps $\W_{G}$ to $\W_{H}$. Given $x\in\Wfin$ we have that $\res_H^Gx\in\W_{H}$ if and only if for each finite subgroup $H'\subseteq H$ the $kH'$--module $\res^{H}_{H'} \res^G_{H}(x)=\res^G_{H'}(x)$ belongs in $\mathrm{Cof}(kH')^\perp$, which clearly holds since $x\in\Wfin$. Hence $\res^G_H$ is right Quillen. It follows formally that $\mathrm{ind}^G_H$ is left Quillen, which by exactness again amounts to $\mathrm{ind}^G_H$ preserving proper cofibrant modules.

In order to prove that $(\mathrm{res}_H^G,\mathrm{coind}_H^G)$ is a Quillen adjunction, for similar reasons as above, it suffices to prove that $\mathrm{res}_H^G$ maps $\mathrm{PCof}(kG)$ to $\mathrm{PCof}(kH)$ and preserves projectives (the latter property is clear since $kG$ is a free $kH$--module).


    

   We first claim that $\res^G_H$ maps into $\mathrm{PCof}(kH)$ all objects of the form $\mathrm{ind}_F^G(x)$ where $F$ is a finite subgroup of $G$ and $x\in\mathrm{Cof}(kF)$, that is, $x$ is a $k$--projective $kF$--module. Indeed,
    by the double coset formula, we have that 
    \[
    \res^G_H\ind_F^Gx\simeq \bigoplus_{g\in H\backslash G/F} \ind_{{}^g\!F\cap H}^H \res^{{}^g\!F}_{{}^g\!F\cap H}({}^g\!x).
    \]
    Clearly, each $\res^{{}^g\!F}_{{}^g\!F\cap H}({}^g\!x)$ is $k$-projective and hence it lies in $\Cof(k{}^g\!F\cap H)$. By 
    the above observation that induction functors preserve proper cofibrant modules we get that 
    \[
    \ind_{{}^g\!F\cap H}^H \res^{{}^g\!F}_{{}^g\!F\cap H}({}^g\!x)\quad \mbox{lies in}\,\,\,\,\mathrm{PCof}(kH)
    \]
     It follows that $\res^G_H\ind_F^Gx $ lies in $\mathrm{PCof}(kH)$ as well. In particular, we have that $\res^G_H(\mathcal{S}_G)\subseteq \mathrm{PCof}(kH)$, where by $\mathcal{S}_G$ we denote the generating set of the complete cotorsion pair $(\mathrm{PCof}(kG),\mathrm{PCof}(kG)^{\bot})$, as in \cref{coro-wfin-complete}.

    Finally, using that $\res^G_H$ is exact and colimit-preserving, we get that for any $x \in \mathrm{PCof}(kG)= {}^\oplus\mathrm{Filt}(\S_G)$, the following holds:
    \[
    \res^G_Hx\in {}^\oplus \mathrm{Filt}(\res^G_H\mathcal S_G)\subseteq {}^\oplus \mathrm{Filt}(\mathrm{PCof}(kH))=\mathrm{PCof}(kH)
    \] 
where the last equality holds by Eklof's lemma--see for instance \cite[Lemma~1]{Eklof}--and the fact that $\mathrm{PCof}(kH)$ is closed under summands. It follows that $\res^G_H$ is a left Quillen functor, which  concludes the proof.
\end{proof}

\begin{remark}
   We note that, in general, it is not clear whether a Benson cofibrant module restricts to a Benson cofibrant module in full generality. However, as shown in the previous lemma, this does hold for proper cofibrant modules. This phenomenon is somewhat surprising, at least to the authors.
\end{remark}

\begin{remark}
    As a consequence of the previous lemma, we get that the restriction functor descends to an exact functor on homotopy categories that preserves both small coproducts and small products. We keep referring to the induced functor as restriction functor and keep the same notation. 
\end{remark} 

We conclude this section with a few observations.

\begin{corollary}\label{coro-detection}
 Let $G$ be a group and $k$ be a commutative ring. Then the family of coproduct-preserving exact (triangulated) functors 
 \[
 (\res^G_F \colon \mathrm{StMod}^{\mathrm{prop}}_{kG}\to \mathrm{StMod}^{\mathrm{prop}}_{kF})_{F\in \Fin(G)}
 \]
 is jointly-conservative, that is, $x\simeq 0$ in the proper stable {module} category of $G$ if and only if $\res^G_Fx\simeq 0$ in the proper stable {module} category of $F$, for all $F\in \Fin(G)$. 
\end{corollary}

\begin{proof}
By the description of the homotopy category of the proper model structure given in \cref{thm-proper-coto}, it suffices to prove that $x\simeq 0$ in $\mathsf{St}\mathrm{PCof}(kG)$ if and only if $\res^G_F(x)\simeq 0$ in  $\mathsf{St}\mathrm{PCof}(kF)$, for all $F\in \Fin(G)$. We have that $x\simeq 0$ in $\mathsf{St}\mathrm{PCof}(kG)$ if and only if $x$ belongs in $\mathrm{Proj}(kG)=\mathcal{W}_{G}\cap {}^\perp\mathcal{W}_{G}$, as a $kG$--module. By \cref{rem:finite_W}, this is equivalent to $\res^G_F(x)\in\mathcal{W}_{F}\cap {}^\perp\mathcal{W}_{F}=\mathrm{Proj}(kF)$ for all finite subgroups $F\leqslant  G$, which finishes the proof.
\end{proof}

\begin{corollary}
    \label{lemma-weak-equivalences}
    Let $f\colon x\to y$ be a morphism of $kG$--modules. Then $f$ is a trivial fibration if and only if $\res^G_Ff$ is a trivial fibration in $\Mod(kF)$ for every $F\in \Fin(G)$. 
\end{corollary}

\begin{proof}
  Assume that $f$ is a trivial fibration in the proper model category of $G$, that is, $f$ is an epimorphism with kernel in $\Wfin$. By  \cref{rem:finite_W} and the fact that restriction functors are exact, we obtain that $\res^G_Ff$ is a trivial fibration in $\Mod(kF)$ for every $F\in \Fin(G)$. Conversely, if we assume the latter condition on $f$, by restricting along the trivial subgroup we deduce that $f$ is an epimorphism, with kernel say $K$. Then by assumption, $K\in\W_{\F}$ for all finite subgroups $F\leqslant  G$. Hence $K\in\Wfin$ and $f$ is a trivial fibration.
\end{proof}

\section{Monoidal model structures}

In this section, our goal is to prove that the proper model structure on $\Mod(kG)$ from \cref{thm-proper-coto} is compatible with the symmetric monoidal structure  on the category $\Mod(kG)$ from \cref{rec-adjuntions-mod}. Explicitly, we aim to prove the following result: 


\begin{theorem}
\label{prop-monoidal-proper}
The proper model structure on $\Mod(kG)$ makes $(\Mod(kG), \otimes,k)$ into a symmetric monoidal model category in the sense of \cite[Definition 4.2.6]{Hov99}.
\end{theorem}

We need the following result from \cite[Proposition 7.2]{Hov02} which gives conditions for an abelian model category to be monoidal.

\begin{theorem}\label{thm-hov}
   Let $\mathcal{A}$ be a complete and cocomplete abelian category  with a symmetric monoidal structure. Suppose that $\mathcal{A}$ is endowed with an abelian model structure satisfying the following properties:
   \begin{enumerate}
       \item For any cofibrant replacement $Q\mathbb 1\to \mathbb 1 $, and any cofibrant object $X$, the canonical map  $X\otimes Q\mathbb 1\to X\otimes \mathbb 1 \simeq X$ is a weak equivalence. 
       \item Tensoring two cofibrant objects produces a cofibrant object.  Moreover, tensoring a cofibrant with a trivially cofibrant produces a trivially cofibrant.
       \item If $f\colon X\to Y$ is a cofibration, and $Z\in \mathcal{A}$, then $Z\otimes f\colon Z\otimes X\to Z\otimes Y$ is a monomorphism. 
   \end{enumerate}
Then, $\mathcal{A}$ is a monoidal model category. 
\end{theorem}

We first need to deal with finite groups. 

\begin{proposition}\label{prop-monoidal-finite}
    Let $G$ be a finite group and $k$ be an arbitrary commutative ring. Then $(\Mod(kG),\otimes, k)$ is a symmetric monoidal model category where $\Mod(kG)$ is equipped with the model structure of  \cref{thm-proper-coto}.
\end{proposition}

\begin{proof}
Since $G$ is finite from \cref{rem:finite_W} we know that the proper model structure on $\Mod(kG)$ is simply the model structure induced by the Hovey triple 
    \[
    (\Cof(kG),\Cof(kG)^\perp,\Mod(kG)).
    \]
    Moreover, we already now that the cofibrant objects are precisely those $kG$--modules that are $k$--projective; see \cref{prop:cof_proj}. 

    In particular, $k$ is cofibrant, and hence condition (a) from \cref{thm-hov} holds. Now, (b) follows by \cref{prop-tensor-proje-latti}, and (c) is a direct consequence of the fact that any cofibration is a monomorphism with cofibrant cokernel, thus fits into a $k$-split short exact sequence. Hence tensoring a cofibration with any $kG$--module produces a monomorphism. This completes the proof. 
\end{proof}

We will also need the following technical result.

 \begin{lemma}
\label{lem:tensors}
Let $x\in\mathrm{PCof}(kG)$ and let $y$ be a $k$--projective $kG$--module. Then $x\otimes y\in\mathrm{PCof}(kG).$ 
\end{lemma}

 \begin{proof}
For any $w\in\mathcal{W}_G$ we want to prove the vanishing of 
\[
\mathrm{Ext}^1_{kG}(x\otimes y,w)\cong \mathrm{Ext}^1_{kG}(x,\hom(y,w)).
\]
For this, it suffices to show that for any finite subgroup $H\leqslant  G$ the $kH$--module $\mathrm{res}^G_H\hom(y,w)$ belongs in $\mathrm{Cof}({kH})^{\perp}$. By definition of the $G$-action on $\hom(y,w)$, it is clear that $\res^G_F(\hom(y,w))$ agrees with $\hom(\res^G_Fy,\res^G_F w)$ as $kF$--modules. Now, for any $\sigma\in\mathrm{Cof}(kH)$ we have an isomorphism $$\Ext^1_{kH}(\sigma,\hom(\res^G_Hy,\res^G_Hw))\cong \Ext^1_{kH}(\sigma \otimes\res^G_Hy,\res^G_Hw).$$
These $\Ext^1$ groups vanish since $\sigma\otimes\res^G_Hy$ belongs in $\mathrm{Cof}(kH)$ as a tensor product of lattices and $\res^G_Hw$ belongs in $\W_H$ by \cref{rem:finite_W}. Thus we have proved that $x\otimes y\in\mathrm{PCof}(kG)$. 
%
%
\end{proof}


\begin{proof}[Proof of \cref{prop-monoidal-proper}] 
We will verify the conditions of \cref{thm-hov}. To prove (c), note again that any cofibration $f$ fits into a $k$-split short exact sequence, hence $f\otimes x$ is a monomorphism for any $kG$--module $x$. 

To prove (b), let $x,y\in \mathrm{PCof}(kG)$ and let $z\in \mathrm{PCof}(kG)\cap \Wfin=\mathrm{Proj}(kG)$. Since $x$ is $k$--projective by \cref{rem-Wfin-contains-cof}, we deduce from \cref{prop-tensor-proje-latti} that $x\otimes z$ is $kG$--projective. On the other hand, by \cref{lem:tensors} we deduce that $x\otimes y \in\mathrm{PCof}(kG)$.

To prove (a), fix a cofibrant replacement $\alpha\colon Qk\to k$ and an object $x$ in $\mathrm{PCof}(kG)$. The morphism $\alpha$ is a trivial fibration in the abelian model structure $\Mod(kG)$, so $\alpha$ is an epimorphism with kernel in $\Wfin$. We want to show that the epimorphism 
\[
x\otimes \alpha\colon x\otimes Qk\to x\otimes k\simeq x
\]
is a weak equivalence, equivalently, a trivial fibration in the proper stable module category of $G$.
By \cref{lemma-weak-equivalences},  trivial fibrations are detected upon restriction to finite subgroups. Let $F\in \Fin(G)$. Since 
\[
\res^G_F \colon \mathrm{StMod}^{\mathrm{prop}}_{kG}\to \mathrm{StMod}^{\mathrm{prop}}_{kF}
\]
is right Quillen by \cref{lemma-Quillen-adj}, it follows that $\res^G_F(\alpha)\colon\res^G_F(Qk)\rightarrow \res^G_F(k)$ is a trivial fibration, which in fact splits over $k$ since $\res^G_F(k)$ is a cofibrant $kF$--module (see \cref{prop:cof_proj}). 
Thus, applying $\res^G_F(x)\otimes-$ produces the epimorphism $\res^G_F(x)\otimes\res^G_F(\alpha)$ with kernel $\res^G_F(x)\otimes\res^G_F(\mathrm{ker}(\alpha))$.
Now, from \cref{rem-Wfin-contains-cof} and \cref{prop:cof_ind_kes} we have that $\res^G_F(x)\in\Cof(kF)$. In addition, we have that $\res^G_F(\mathrm{ker}(\alpha))\in\W_{F}\cap \Cof(kF)$; indeed, this  follows by \cref{rem:finite_W} and the fact that both $\res^G_F(Qk)$ and $\res^G_F(k)$ are contained in $\Cof(kF)$, with the latter class being closed under kernels of epimorphisms. Hence by what was proved above on part (b) we deduce that $\res^G_F(x)\otimes\res^G_F(\mathrm{ker}(\alpha))$ belongs in $\W_{F}\cap \Cof(kF)$, which proves that $\res^G_F(x)\otimes\res^G_F(\alpha)$ is a trivial fibration. Lastly, we notice an isomorphism $\res^G_F(x)\otimes\res^G_F(\alpha)\cong \res^G_F(x\otimes \alpha)$ as maps, which concludes the proof.
\end{proof}

As an immediate consequence we obtain the following result.

\begin{corollary}\label{coro-proper-ttcat}
     The proper stable module category $\mathrm{StMod}^{\mathrm{prop}}_{kG}$ is a tensor-triangulated category. 
\end{corollary}

\section{A regular base}

Recall that a commutative ring $k$ is called regular if it is Noetherian and $k_\pp$ has finite global dimension for all primes $\pp$ of $k$. Our main result of this section is the following. 

\begin{theorem}\label{thm-proper-compact} 
 Let $k$ be a regular commutative ring and $G$ be a group. Then the proper stable module category $\mathrm{StMod}^{\mathrm{prop}}_{kG}$ is compactly generated tensor-triangulated category. 
\end{theorem}

We need the following abstract criteria to detect compact generation. 

\begin{lemma}\label{lemma-compact-abstract}
   Let $(f^\ast_i\colon \T\to \S_i)_{i\in I}$ be a family of coproduct preserving exact functors between triangulated categories with small coproducts. Assume that the following properties hold:
   \begin{enumerate}
       \item  $\S_i$ is compactly generated for all $i\in I$; with the subcategory of compact objects denoted by $\S_i^c$.
       \item For each $i\in I$, the functor $f^\ast_i$ has a left adjoint $(f_i)_!$.   
   \end{enumerate}
   Then $\mathrm{loc}((f_i)_!(\S_i^c)\mid i\in I)=\T$ if and only if $(f_i^\ast)_{i\in I}$ is jointly conservative; that is, $y\simeq 0$ if and only if $f_i^*(y)\simeq 0$ for all $i\in I$. Note that in this case $\T$ is compactly generated, and
   \[
   \{ (f_i)_!(\mathcal{X}_i)\mid \mathcal{X}_i \mbox{ a set of compact generators in } \S_i, \, i\in I\}
   \]
   is a set of compact generators. 
\end{lemma}

\begin{proof}
By assumption the various $(f_i)_!$ preserve compact objects, so that for all $i\in I$ we have that $(f_i)_!(\mathcal{S}_i^c)\subseteq \T^c$. 
The condition$\mathrm{loc}((f_i)_!(\S_i^c)\mid i\in I)=\T$ is equivalent to the following property: for any $y\in \T$ we have that $y\simeq 0$ if and only if 
\[
\forall x\in \{(f_i)_!(\S_i^c)\,|\, i\in I\}\,\, \mbox{we have that}\, \Hom_\T(x,y)=0.
\] 
Since $x$ must have the form $(f_i)_!(z)$ for some $i\in I $ and some $z\in \S_i^c$, we can rewrite this as 
\begin{align*}
y\simeq 0 \iff &  \forall z\in \S_i^c, \forall i\in I,\,\,  \Hom_\T((f_i)_!(z),y)=0\\ 
 \iff & \forall z\in \S_i^c, \forall i\in I,\,\, \Hom_{\S_i}(z, f^\ast y)=0 \\ 
 \iff & \forall i\in I,\,\, f^\ast_i(y)\simeq 0
\end{align*}    
where the last equivalence follows from assumption $(a)$. This of course implies that $\T$ is compactly generated since the $(f_i)_!$ preserve compact objects.  
\end{proof}

\begin{proof}[Proof of \cref{thm-proper-compact}]
    We already know that $\mathrm{StMod}^{\mathrm{prop}}_{kG}$ is a tt-category with small coproducts by \cref{coro-proper-ttcat}. It remains to show that it is compactly generated. Consider the family of restriction functors 
\[
 (\res^G_F \colon \mathrm{StMod}^{\mathrm{prop}}_{kG}\to \mathrm{StMod}^{\mathrm{prop}}_{kF})_{F\in \Fin(G)}
 \]
 which is jointly-conservative by \cref{coro-detection}. Moreover, we have that the restriction functor has a left adjoint given by induction; see \cref{lemma-Quillen-adj}. Now, recall that the proper model structure on $\Mod(kF)$ agrees with model structure obtained from the cotorsion pair $(\Cof(kF),\Cof(kF)^\perp)$, and hence $\mathrm{StMod}^{\mathrm{prop}}_{kF}$ agrees with the stable category of lattices $\mathrm{StMod}(kF,k)$ coming from \cref{cor-frobenius}. Since the ring $k$ is regular, it also agrees with $\mathrm{StMod}^{\mathrm{lf}}(kF,k)$ from \cref{rec-modlf-comp},   which is compactly generated by \cref{rec-regular-ring}. We now invoke \cref{lemma-compact-abstract} to deduce the result. 
\end{proof}

\section{The proper stable $\infty$-category}\label{sec-infty}

In  this section, we introduce the proper stable $\infty$-category which is an $\infty$-categorical enhancement of the proper stable module category introduced in \cref{sec-proper}. We also record some properties of this category. We refer to \cite{Lur17} and \cite{Cis19} for an account on the theory of stable $\infty$-categories, as well as to some surveys on this topic \cite{Ant16} and \cite{Jas26}. 

\begin{recollection} 
    Let $\A$ be a complete and cocomplete category equipped with  a (closed) model structure $(\mathrm{Cofib}, \mathrm{W},\mathrm{Fib})$. The \textit{underlying $\infty$-category}  
    \[
    \mathrm{L}_\mathrm{W}\A\coloneqq \A[\mathrm{W}^{-1}]
    \]
    of $\A$ is the $\infty$-categorical localization of $\A$ at the class of weak equivalences $\mathrm{W}$; see \cite[Definition 1.3.4.1]{Lur17}. In this case, we can identify the homotopy category of $\mathrm{L}_\mathrm{W}\A$ with the 1-categorical
localization of $\A$ at the class of weak equivalences \cite[Remark 7.1.6]{Cis19}. Moreover, by \cite[Remark 1.3.4.16]{Lur17}, the underlying $\infty$-category of $\A$ can equivalently be obtained as the $\infty$-categorical localization of the full subcategory of fibrant-cofibrant objects at the class of weak equivalences. That is, 
\[
\mathrm{L}_\mathrm{W}(\C\cap \F) \xrightarrow[]{\simeq} \mathrm{L}_\mathrm{W}\A
\]
where $\C$ and $\F$ denote the subcategories of cofibrant and fibrant objects, respectively. 
\end{recollection}

\begin{definition}\label{def-inf-proper}
   Let $G$ be a group and $k$ be a commutative ring.  We define the \textit{proper stable module $\infty$-category} 
   \[
   \mathbf{StMod}(kG)\coloneqq \mathrm{L}_{\Wfin} \Mod(kG)
   \]
   as the underlying $\infty$-category of $\Mod(kG)$ equipped with the model structure from \cref{thm-proper-coto}, with its weak equivalences denoted by $\mathcal{W}_G$.
\end{definition}

\begin{remark}
    By definition, we obtain that the homotopy category of $ \mathbf{StMod}(kG)$ agrees with the proper stable module category $\mathrm{StMod}^{\mathrm{prop}}_{kG}$. 
\end{remark}

\begin{remark}
It will be shown later in \cref{cor-ms-stable-agrees} and \cref{rem-stable-inftycat2} that our proper stable module $\infty$-category $\mathbf{StMod}(kG)$ recovers the stable module $\infty$-category introduced in \cite{Gom24}. In fact, just as in \textit{op. cit.}, we can use results already proved to draw some immediate consequences for the proper stable module $\infty$-category.   
\end{remark}

\begin{proposition}\label{prop-stmod-is-sht}
     Let $G$ be a group and $k$ be a commutative ring. Then $ \mathbf{StMod}(kG)$, with monoidal structure induced from \cref{prop-monoidal-proper}, is a presentable, symmetric monoidal stable $\infty$-category, where the tensor product commutes with colimits in each variable.
\end{proposition}

\begin{proof}
    By \cref{prop-monoidal-proper} and \cref{thm-proper-coto} we have that $\Mod(kG)$ is a
combinatorial stable symmetric monoidal model category. It follows that $\mathbf{StMod}(kG)$ is
presentable and symmetric monoidal by \cite[Proposition 1.3.4.22]{Lur17} and \cite[Example 4.1.7.6]{Lur17}. Additionally, the tensor product on $\mathbf{StMod}(kG)$ commutes
with colimits separately in each variable by \cite[Lemma 4.1.8.8]{Lur17}. 
\end{proof}

We now record some properties about functors induced along a monomorphism of groups.



\begin{corollary}\label{coro-functors-propstable}
    Let $i\colon H\to  G$ be a group monomorphism. Then we have  adjunctions between proper stable module $\infty$-categories 
    \begin{center}
          \begin{tikzcd}[column sep=large, row sep=large]
     \mathbf{StMod}(kH) \arrow[r,"\ind_i", yshift=3mm] \arrow[r, yshift=-3mm, "\mathrm{coind}_i"'] & \mathbf{StMod}(kG). \arrow[l, "\res_i" description] 
    \end{tikzcd}
    \end{center}
    Moreover, the functor  $\res_i$ is strongly monoidal. In particular, we obtain a symmetric monoidal functor 
    \[
    c^\ast_g\colon \mathbf{StMod}(kH)\to \mathbf{StMod}(k{}^g\!H)
    \]
    induced by the conjugation map $c_g^\ast\colon {{}^g\!H}\to H$, for any $g\in G$. 
\end{corollary}

\begin{proof}
    The Quillen adjunction from \cref{lemma-Quillen-adj} lifts to the $\infty$-categorical setting; see \cite{Maz16}. The fact that $\res_i$ is strongly monoidal follows since $\res_i$ is a symmetric monoidal left Quillen functor between the monoidal structures defining the proper stable module categories. This follows from \cite[Example 4.1.7.6]{Lur17}; see also \cite{NS17} for a more detailed discussion. 
\end{proof}

\subsection{Functoriality and descent} \label{sec-functoriality}

In this section, we investigate the functoriality properties on the group variable of the proper stable module $\infty$-category and conclude with a descent result. We thank Luca Pol for explaining to us some of the arguments involving the functoriality.

The following recollections are from \cite{GP26}.

\begin{recollection}
    Let $G$ be a discrete group and $\F$ be a collection of subgroups of $G$. 
    We write $\O_{G,\F}$ to denote the orbit $\infty$-category of $G$ whose objects are $G$-sets of the form $G/H$ for $H\in \F$ and whose morphisms are given by $G$-equivariant maps. When $\F$ is the collection of all subgroups we simple write $\O_G$. In particular, for any collection $\F$ we have a canonical functor 
    \begin{align}
        \label{eq-iota}
        \iota \colon \O_{G,\F} \to \O_G
    \end{align}
    From here, given an $\infty$-category $\D$,  for any functor $f\colon \O_G^{\nameop}\to \D$ we obtain a restricted functor $f\circ \iota\colon \O_{G,\F}^{\nameop} \to \C$.  If the context is clear, we will omit the functor $\iota$ from the notation to refer to a functor obtained via precomposition with $\iota$.
\end{recollection}

\begin{notation}  
    We write $\PrL$ to denote the $\infty$-category of presentable stable $\infty$-categories and left adjoint functors, and $\CAlg(\PrL)$ to denote the $\infty$-category  of presentable symmetric monoidal stable $\infty$-category satisfying that the tensor product is coproduct preserving in each variable separately and symmetric monoidal left adjoint functors. 
\end{notation}

\begin{remark}\label{rem-functor-mod} 
    Let $G$ be a group and $k$ be a commutative ring.  As observed in the first part of  \cite[Construction 6.5]{GP26}, there is a functor 
    \[
    \Mod(k-)\colon \O_G^{\nameop} \to \mathrm{Cat}^\otimes 
    \]
    which maps a coset $G/H$ to the symmetric monoidal abelian category $\Mod(kH)$, and any $G$-equivariant map $G/H\xrightarrow[]{g} G/H' $ to the restriction functor along the conjugation map $c^\ast_g$. Here the right hand side denotes the category of symmetric monoidal categories. 
    
    We can restrict this functor to a functor from $\O_{G,\F}^{\nameop}$ for any collection of subgroups $\F$ of $G$.
\end{remark}

\begin{recollection}
    Let $\W\mathrm{Cat}_\infty $  be the   $\infty$-category of pairs $(\C, W)$ where $\C$ is an $\infty$-category, and $W$ is a collection of morphisms which is stable under homotopy, composition, and contains all equivalences; that is, $W$ is a subcategory of the homotopy category  $\mathrm{h}\C$ that contains all objects and isomorphisms in $\mathrm{h}\C$. A map in $\W\mathrm{Cat}_\infty$ between $(\C, W) \to (\C', W')$ is a functor of $\infty$-categories $f\colon \C \to \C'$ such that $f(W)\subseteq W'$; see \cite[Construction 4.1.7.1]{Lur17} for further details. 
   In particular, there is a finite product preserving localization functor 
    \begin{equation}\label{eq-localization}
   \W\mathrm{Cat}_\infty \to \mathrm{Cat}_\infty, \quad (\C,W)\mapsto \C[W^{-1}]
    \end{equation}
    see  \cite[Proposition 4.1.7.2]{Lur17}.
\end{recollection}

\begin{construction}\label{const-stmod-functor}
    Consider the abelian category $\Mod(kG)$ equipped with the abelian model structure from \cref{thm-proper-coto}. In this case, we can promote the functor from \cref{rem-functor-mod} to a functor 
     \[
    (\Mod(k-), \W_{-}) \colon \O_G^{\nameop} \to \W\mathrm{Cat}_\infty 
    \]
    which sends a coset $G/H$ to the pair $(\Mod(kH),\W_{H})$, and any equivariant map $G/H\xrightarrow[]{g} G/H' $ to the Quillen functor corresponding to the restriction functor along the conjugation map $c^\ast_g$. Now, composing with the functor from \cref{eq-localization} we obtain a functor 
     \[
    \mathbf{StMod}(k-) \colon \O_G^{\nameop} \to \mathrm{Cat}_\infty 
    \] 
    which maps a coset $G/H$ to the proper stable module $\infty$-category $\mathbf{StMod}(kH)=\Mod(kH)[\W_{H}^{-1}]$.  Finally,  by \cref{prop-stmod-is-sht} and \cref{coro-functors-propstable},  we deduce that the functor $\mathbf{StMod}(k-)$ actually takes values in $\CAlg(\PrL)$. 
\end{construction}

\begin{notation}\label{notation-orb}
    Let $\C_\bullet\colon \O_G^{\nameop}\to \CAlg(\PrL)$ be a functor. Then it admits a unique limit-preserving lift to a functor 
    \[
   \C_\bullet\colon \S_{G}^{\nameop} \to \CAlg(\PrL)
    \]
    where $\S_G$ denotes the presheaf category on $\O_G$; that is, the $\infty$-category of functors from $\O_G^{\mathrm{op}}$ to the $\infty$-category of $\infty$-groupoids. See \cite[Remark 5.3.5.9]{Lur09}.  Note that we keep the same notation for the functor $\C_\bullet$ and its extension. 
\end{notation}


\begin{proposition}
      Let $G$ be a group and $k$ be a commutative ring. Let 
      \[
      \mathbf{StMod}(k-)\colon \O_G^{\nameop} \to \CAlg(\PrL) 
      \]
      be the functor from \cref{const-stmod-functor}. Then the following properties hold.
    \begin{enumerate}
        \item For any $G$-equivariant map $f \colon G/K \to G/H$, the corresponding functor $f^* \colon \mathbf{StMod}(kH) \to \mathbf{StMod}(kK)$ admits a left adjoint $f_!$. When $f$ is induced by an inclusion we will write $f^*=\res^H_K$ and $f_!=\ind_K^H$.
        \item For any pullback square in $\S_{G}$
    \[
    \begin{tikzcd}
        X \arrow[r] \arrow[d] & G/L\arrow[d] \\
        G/H \arrow[r] & G/K
    \end{tikzcd}
    \]
  where we identify elements in $\O_G$ with elements in $\S_G$ via the Yoneda embedding, the corresponding diagram 
    \[
    \begin{tikzcd}
        \mathbf{StMod}(kK)\arrow[r] \arrow[d] & \mathbf{StMod}(kH)\arrow[d] \\
        \mathbf{StMod}(kL) \arrow[r] & \mathbf{StMod}(kX)
    \end{tikzcd}
    \]
    is horizontally left adjointable; see \cite[Definition 4.7.4.13]{Lur17}. 
    \item The collection of functors $(\res^G_F)_{F \in  \Fin(G)}$ is jointly conservative. 
    \end{enumerate}
\end{proposition}

\begin{proof}
    Part (a) follows from \cref{coro-functors-propstable}. Part (b) corresponds to the double coset formula: since this holds at the abelian categorical level, we obtain the corresponding result for proper stable module $\infty$-categories. Part (c) is a consequence of \cref{coro-detection} since this property is detected at the level of homotopy categories. 
\end{proof}

As an immediate consequence of \cite[Theorem 7.4]{GP26} combined with the previous proposition, we obtain the following result:  

\begin{corollary}\label{coro-stmod-descent}
    Let $G$ be a group and $k$ be a commutative ring. Then the restriction functors induce an equivalence of symmetric monoidal stable $\infty$-categories 
    \[
    \mathbf{StMod}(kG)\xrightarrow[]{\simeq} \lim_{G/F\in \Ofin^{\nameop}}\mathbf{StMod}(kF).
    \]
\end{corollary}

\section{Proper cofibrants and Benson cofibrants}
\label{subsec:comparison} 


We now compare the proper cofibrant $kG$--modules $\mathrm{PCof}(kG)$ with the Benson cofibrant $kG$--modules $\mathrm{Cof}(kG)$.
We recall that $\mathrm{PCof}(kG)\subseteq \mathrm{Cof}(kG)$; see \cref{rem-Wfin-contains-cof}. The next result is akin to \cite[Proposition~4.3]{ET_cof}.

\begin{proposition}
\label{prop:equality of classes}
The following are equivalent:
\begin{itemize}
\item[(i)] $\mathrm{Cof}(kG)=\mathrm{PCof}(kG)$
\item[(ii)] $\mathrm{Cof}(kG)\cap\Wfin = \mathrm{Proj}(kG)$.
\end{itemize}
\end{proposition}

\begin{proof}
$(i)\Rightarrow(ii)$ Observe that $\mathrm{Cof}(kG)\cap\Wfin={}^\perp\Wfin\cap\Wfin = \mathrm{Proj}(kG)$.

$(ii)\Rightarrow(i)$ Let $x\in\mathrm{Cof}(kG)$. Part of \cref{thm-proper-coto} implies that there exists a short exact sequence $0\rightarrow z\rightarrow t\rightarrow x\rightarrow 0$ where $t\in \mathrm{PCof}(kG)$ and $z\in\Wfin$. In particular, $t\in\mathrm{Cof}(kG)$, thus $z\in\mathrm{Cof}(kG)$ too since Benson cofibrants are closed under kernels of epimorphisms; see \cref{rem:properties of cofs}. But then $z\in\mathrm{Proj}(kG)$ by assumption. In particular $z\in\mathrm{Cof}(kG)^{\perp}$; thus $x$ is a summand of $t$. Since $\mathrm{PCof}(kG)$ is closed under summands, the result follows.
\end{proof}

\begin{definition}
A $kG$--module $x$ is \textit{locally projective} if $\res^G_F(x)$ is a projective $kF$--module for all finite subgroups $F\leqslant G$. We denote by $\mathcal{F}\mathrm{Proj}(kG)$ the class of locally projective modules.    
\end{definition}

\begin{remark}
Since the restriction of a cofibrant $kG$-module along a finite subgroup is cofibrant (\cref{prop:cof_ind_kes}) and by definition of $\Wfin$, we deduce that 
\begin{equation}
\label{eq:equality with fproj}
\mathrm{Cof}(kG)\cap\Wfin=\mathrm{Cof}(kG)\cap\mathcal{F}\mathrm{Proj}(kG).
\end{equation}
Hence, \cref{prop:equality of classes} tells us that $\mathrm{Cof}(kG)=\mathrm{PCof}(kG)$ if and only if all the cofibrant and locally projective $kG$--modules are projective.
\end{remark}

\begin{remark}
\cref{prop:equality of classes} hints at a connection between the stable categories of $\mathrm{PCof}(kG)$ and $\Cof(kG)$ that we now explore. Firstly, notice that both categories admit the canonical exact structure induced from the  category of all $kG$--modules, hence  $\mathrm{PCof(kG)}$ is an extension-closed subcategory of $\Cof(kG)$. Both are Frobenius exact categories with projective-injective objects given by the projective $kG$--modules. It follows that the inclusion $\mathrm{PCof}(kG)\subseteq \Cof(kG)$ descends to a canonical fully faithfull exact functor 
    \begin{equation}
    \label{eq:exact_functor}
   j\colon  \mathsf{St}\, \mathrm{PCof}(kG)\to \mathsf{St}\,\mathrm{Cof}(kG).
\end{equation}
\end{remark}

We claim that the canonical functor in \eqref{eq:exact_functor} is part of a localization sequence of triangulated categories. For this we will use the following facts concenring Bousfield localizations, in the context of Hovey triples.

\begin{recollection}\textnormal{(\cite[Proposition~1.4.6]{Bec14})}
\label{recoll:Bousfield}
Given a complete and cocomplete abelian category $\mathcal{M}$ admitting two hereditary Hovey triples
$\mathcal{M}_1=(\C_1,\W_1,\mathcal{M})$ and $\mathcal{M}_2=(\C_2,\W_2,\mathcal{M})$ such that $\C_2\subseteq \C_1$, there exists a new hereditary Hovey triple on $\mathcal{M}$ given by 
$\mathcal{M}_2\setminus \mathcal{M}_1\coloneqq(\C_1,\mathcal{T},\W_2)$ where 
\begin{eqnarray}
\T & \coloneqq & \{M\in\mathcal{M}\, |\, \exists\,\,\, \mbox{ex. seq.}\,\,\,  0\rightarrow M\rightarrow A\rightarrow B\rightarrow 0\,\,\, \mbox{where}\,\,\,  A\in\W_1,\,  B\in\C_2\} \nonumber \\
 & = &  \{M\in\mathcal{M}\, |\, \exists\,\,\, \mbox{ex. seq.}\,\,\,  0\rightarrow A\rightarrow B\rightarrow M\rightarrow 0\,\,\, \mbox{where}\,\,\,  A\in\W_1,\,  B\in\C_2\}. \nonumber
\end{eqnarray}
The Hovey triple $\mathcal{M}_2\setminus \mathcal{M}_1$ is the \textit{left (Bousfield) localization} of $\mathcal{M}_1$ with respect to $\mathcal{M}_2$. In this situation, there exists a localization sequence of triangulated categories:
\begin{equation}
\label{eq:Becker_loc}
\xymatrix@C=3pc{
\mathbf{Ho}(\mathcal{M}_2) \ar[r]^-{\mathbf{L}(\mathrm{id})} & \mathbf{Ho}(\mathcal{M}_1) \ar[r]^-{\mathbf{L}(\mathrm{id})} & \mathbf{Ho}(\mathcal{M}_2\setminus \mathcal{M}_1); 
}
\end{equation}
\end{recollection}

\begin{theorem}
\label{thm:comparison}
There exists a localization sequence of triangulated categories
\begin{equation}
\label{eq:loc_seq}
\xymatrix@C=3pc{
\mathsf{St}\,\mathrm{PCof}(kG)  \ar[r]^-{j} & \mathsf{St}\,\mathrm{Cof}(kG) \ar[r]^-{\pi} & \mathsf{St}({\mathrm{Cof}(kG)\cap\mathcal{F}\mathrm{Proj}}(kG)). 
}
\end{equation}
In particular, the canonical fully faithful functor $j\colon \mathsf{St}\,\mathrm{PCof}(kG)\rightarrow \mathsf{St}\,\mathrm{Cof}(kG)$  is an equivalence of triangulated categories if and only if all cofibrant and locally projective $kG$--modules are projective.
\end{theorem}

\begin{proof}
Consider the hereditary Hovey triples $\mathcal{M}_1=(\Cof(kG),\Cof(kG)^\perp,\Mod(kG))$ and $\mathcal{M}_2=(\mathrm{PCof}(kG),\mathrm{PCof}(kG)^{\bot},\Mod(kG))$ from \cref{prop-model-struc} and  \cref{thm-proper-coto} respectively. Recall from \cref{rem-Wfin-contains-cof} that $\mathrm{PCof}(kG)\subseteq \mathrm{Cof}(kG)$ and notice that the two model structures have the same class of fibrant objects, which are all $kG$--modules. Hence, we may apply the results of  \cref{recoll:Bousfield} and obtain a hereditary Hovey triple $\mathcal{M}_2\setminus \mathcal{M}_1=(\mathrm{Cof}(kG),\mathcal{T},\Wfin)$ for a certain class $\mathcal{T}$. The homotopy category of this abelian model structure is a triangulated category equivalent to the stable category of the Frobenius category $\mathrm{Cof}(kG)\cap\Wfin$. The latter category is identified with $\mathrm{Cof}(kG)\cap\mathcal{F}\mathrm{Proj(kG)}$ by \eqref{eq:equality with fproj}. 
Lastly, it remains to observe that the localization sequence \eqref{eq:Becker_loc} restricts to  \eqref{eq:loc_seq}.
\end{proof}


\begin{corollary}\label{cor-comparison}
The  family of exact functors 
    \[
    (\res^G_F\colon \mathsf{St}\Cof(kG)\to \mathsf{St}\Cof(kF))_{F\in \Fin(G)}
    \]
    is jointly-conservative if and only if the canonical functor 
    \[
\mathsf{St}\,\mathrm{PCof}(kG)\to \mathsf{St}\,\mathrm{Cof}(kG))
    \]
    is an equivalence of triangulated categories. 
\end{corollary}

\begin{proof}
First, notice that for any finite subgroup $F\leqslant  G$ the restriction functor $\res^G_F\colon \mathsf{St}\Cof(kG)\to \mathsf{St}\Cof(kF)$ is well-defined by  \cref{prop:cof_proj} and the fact that it sends projective $kG$--modules to projective $kF$--modules. 
The displayed family of exact functors is jointly conservative if, by definition,  any cofibrant $kG$--module $x$ satisfies $x\cong 0$ if and only if $\res^G_F(x)\cong 0$ for all finite subgroups $F\leqslant G$. Equivalently, this condition asks that a cofibrant $kG$--module $x$ is projective if and only if it is locally projective. Hence the result follows by \cref{thm:comparison}.
\end{proof}



\subsection{Hierarchically defined groups}\label{subsec-hiera}

Our aim in this subsection is to investigate the proper stable module category for special classes of group algebras. In particular, in \cref{cor-ms-stable-agrees} we show that the Mazza-Symonds stable category from \cite{MS19} is an instance of our proper stable module category.

 \begin{recollection}\label{rec-krop}
     Let $\mathcal{X}$ be a class of groups. Following Kropholler \cite{Krop} the class of groups $\mathrm{H}\mathcal{X}$ is defined as follows: $\mathrm{H}_0\mathcal{X}\coloneqq \mathcal{X}$ and for all ordinals $\alpha > 0$, a group is contained in $\mathrm{H}_{\alpha}\mathcal{X}$ if it acts cellularly on a finite dimensional contractible CW-complex such that the stabliser of each cell belongs in $\mathrm{H}_{\beta}\mathcal{X}$ for some $\beta < \alpha$. 
     
Then, we say that a group belongs in the class $\mathrm{H}\mathcal{X}$ if it belongs to $\mathrm{H}_{\alpha}\mathcal{X}$, for some ordinal $\alpha$. Finally, $\mathrm{LH}\mathcal{X}$ denotes the class of groups whose finitely generated subgroups belong in various classes of the form $\mathrm{H}_{\alpha}\mathcal{X}$.
 \end{recollection}

In particular, we are interested in the case where $\mathcal{X}$ is either the class of finite groups or the class of groups of type $\Phi_k$. The latter was introduced in \cite{Tal}. For convenience, we recall the definition of this class of groups. 

\begin{definition}
\label{def:phi}
    Let $k$ be a commutative ring. A group $G$ is of type $\Phi_k$ if for any $kG$--module $x$ the following property holds: 
    \[
    \mathrm{pd}_{kG}x<\infty \iff \forall F\in \Fin(G);\,\, \mathrm{pd}_{kF}\res^G_F x  <\infty. 
    \]
Here $\Fin(G)$ denotes the collection of all finite subgroups of $G$.
\end{definition}

We also recall that a $kG$--module $x$ is called \textit{Gorenstein projective} if it is a syzygy in an acyclic complex $P$ of projective $kG$--modules and for any projective module $kG$--module $y$ the complex $\Hom_{kG}(P,y)$ is acyclic too. 

It is known by \cite{CP} that $\mathrm{Cof}(kG)\subseteq \mathrm{GProj}(kG)$ and it is an active research question to understand for which group algebras the two classes are equal.  For instance, it was recently proved that  $\mathrm{Cof}(kG)=\mathrm{GProj}(kG)$ if $k$ has finite weak global dimension and $G$ is of type $\Phi_k$ or an $\mathrm{LH}\mathcal{F}$ group, see \cite[Corollary 2.5]{ET_cof}. 

To compare $\mathrm{GProj}(kG)$ with $\mathrm{PCof}(kG)$ we have the following result, which is analogous to \cref{prop:equality of classes}.

\begin{proposition}
\label{prop:equality of G-classes}
For a commutative ring $k$ and a group $G$ the following are equivalent:
\begin{itemize}
\item[(i)] $\mathrm{GProj}(kG)=\mathrm{PCof}(kG)$
\item[(ii)] $\mathrm{GProj}(kG)\cap\Wfin = \mathrm{Proj}(kG)$.
\end{itemize}
\end{proposition}

\begin{proof}
$(i)\Rightarrow(ii)$ The assumption implies in particular that $\mathrm{GProj}(kG)^{\bot}=\mathcal{W}_G$. Since $\mathrm{GProj}(kG)\cap \mathrm{GProj}(kG)^{\bot}=\mathrm{Proj}(kG)$ the result follows.

$(ii)\Rightarrow(i)$ Since $\mathrm{Cof}(kG)^{\perp}\subseteq \Wfin$, by combining the assumption together with \cite[Proposition 4.3]{ET_cof} we deduce  that $\mathrm{GProj}(kG)=\mathrm{Cof}(kG)$. Thus by \cref{prop:equality of classes} we deduce that $\mathrm{Cof}(kG)=\mathrm{PCof}(kG)$, which concludes the proof.
\end{proof}

\begin{corollary}
\label{cor:when all classes agree}
Let $k$ be a commutative ring and $G$ a group such that $\mathrm{Cof}(kG)=\mathrm{GProj(kG)}$. Then $\mathrm{GProj}(kG)=\mathrm{PCof}(kG)$ if and only if  all locally projective Gorenstein projective $kG$--modules are projective.
\end{corollary}

\begin{proof}
Follows from \cref{prop:equality of G-classes} together with \cref{thm:comparison}.
\end{proof}

\begin{proposition}\label{prop-model-hphi-new}
Let $k$ be a commutative ring and $G$ a group of type $\Phi_k$ such that $\mathrm{Cof}(kG)=\mathrm{GProj(kG)}$. Then $\mathrm{GProj}(kG)=\mathrm{PCof}(kG)$ holds true. In particular, $\mathsf{St}{\mathrm{GProj}}(kG)$ is equivalent to the proper stable module category of $G$.
\end{proposition}

\begin{proof}
According to \cref{cor:when all classes agree} it suffices to prove that the locally projective Gorenstein projective $kG$--modules are projective. By \cref{def:phi} these modules are Gorenstein projective and of finite projective dimension, hence they are projective; see for instance \cite[Proposition 2.27]{Holm}. 
\end{proof}

Under the assumption that $k$ is commutative of finite global dimension, Mazza and Symonds introduced a stable module category in \cite[Definition~3.2]{MS19}. For groups of type $\Phi_k$, this is again the proper stable module category: 

\begin{corollary}\label{cor-ms-stable-agrees}
 Let $k$ be a commutative ring of finite global dimension and let $G$ be a group of type $\Phi_k$. Then the Mazza-Symonds stable module category is equivalent to the proper stable module category.
\end{corollary}

\begin{proof}
Under the assumptions given, by \cite[Corollary~2.5]{ET_cof} we deduce that 
$\mathrm{Cof}(kG)=\mathrm{GProj(kG)}$. Thus \cref{prop-model-hphi-new} implies that $\mathsf{St}{\mathrm{GProj}}(kG)$ is equivalent to the proper stable module category of $G$. Since the Mazza-Symonds stable module category is equivalent to $\mathsf{St}{\mathrm{GProj}}(kG)$ by \cite[Theorem 3.10]{MS19}, the result follows.
\end{proof}


\begin{remark}\label{rem-stable-inftycat2}
Let $k$ be a field and $G$ a group of type $\Phi_k$. In \cite{Gom24} the author builds a model structure on $kG$--modules, whose homotopy category agrees with the Mazza-Symonds stable module category. It follows that this is also an instance of the proper stable module category on $kG$--modules.
\end{remark}

The following result may be seen  as a generalization of \cref{cor-ms-stable-agrees}.

\begin{proposition}\label{prop-model-hphi}
    Let $k$ be a commutative ring of finite global dimension and let $G$ be a group in $\mathrm{H}\Phi_k$. Then $\Wfin$ agrees with $\Cof(kG)^\perp$. Therefore the triples
    \[
    (\mathrm{PCof}(kG),\mathrm{PCof}(kG)^{\bot},\Mod(kG)) \mbox{  and  } (\Cof(kG),\Cof(kG)^\perp,\Mod(kG))
    \]
      induce the same abelian model structure on $\Mod(kG)$.  In particular, $\mathsf{St}{\Cof}(kG)$ is equivalent to the proper stable module category of $G$.
\end{proposition}

\begin{proof}
We proceed by induction on the ordinal $\alpha$ such that $G\in \mathrm{H}_\alpha\Phi_k$. The base case is $\alpha=0$ {and follows from  \cref{cor-ms-stable-agrees}.}


Now, assume that $G$ is in $\mathrm{H}_\alpha\Phi_k$ for $\alpha\geq1$. By definition, there is a finite-dimensional CW-complex $X$ with a cellular action of $G$ such that the isotropy group of each cell lies in $\mathrm{H}_\beta\Phi_k$, for $\beta<\alpha$. In particular, the augmented  cellular chain complex $C_\ast(X;k)$ of $X$ is an acyclic complex of $kG$--modules 
 \begin{equation}
        \label{eq-cell-chain}
        0\to C_n \to \ldots \to C_0\to k\to 0
    \end{equation}
where $C_i$ is of the form $\bigoplus \ind_H^G k$ with $H\in \mathrm{H}_\beta\Phi_k$, for $\beta<\alpha$. In fact, $C_i$ is indexed over the equivariant cells of $X$ of dimension $i$, and the $H$ are the isotropies of such equivariant cells. 

Now, let $y\in \Wfin$.  Note that the complex from \cref{eq-cell-chain} is $k$-split, hence applying the functor $\hom_k(-,y)$ gives us an acyclic complex of $kG$--modules
\[
 0\to y \to \hom_k(C_0,y) \to \ldots \to \hom_k(c_n,y) \to 0
\]
We claim that the terms $\hom_k(C_i,y)$ belong to $\Cof(kG)^\perp$. Note that if this is the case we can conclude the result since $\Cof(kG)^\perp$ has the 2-out-of-3 property. Let $x\in \Cof(kG)$. We have  
 \begin{align*}
         \mathrm{Ext}^1_{kG}(x,\hom_k(C_i,y))\simeq & \mathrm{Ext}^1_{kG}(x\otimes C_i, y) \\ \simeq &   \prod \mathrm{Ext}^1_{kG}(\ind_H^G(\res^G_H(x)),y) \\ 
         \simeq & \prod \mathrm{Ext}^1_{kH}(\res^G_H(x),\res^G_H(y))=0.
    \end{align*}
 The second equivalence follows from the fact that the tensor product commutes with coproducts and  the projection formula: $(\ind_H^Gk)\otimes x \simeq \ind_H^G (\res^G_H x)$. The last equality follows from the induction hypothesis. Indeed, under our assumptions we can invoke \cite[Corollary 2.5]{ET_cof} which says that the class of Benson cofibrants agrees with the class of Gorenstein projective modules. It follows that $\res^G_H x$ must be in $\Cof(kH)$ and hence it must also lie in $\mathrm{PCof}(kH)$ by induction. This completes the proof. 
\end{proof}

\section{Homological dimensions}\label{sec-homological}

In this section we study finiteness conditions of certain homological dimensions defined via the classes of Benson cofibrant and proper cofibrant objects. As usual, we fix a commutative ring $k$ and a group $G$.

\begin{definition} Let $x$ be a $kG$--module.  If there exists an exact sequence of the form 
\[
0\rightarrow c_n\rightarrow c_{n-1}\rightarrow \dots\rightarrow c_0\rightarrow x\rightarrow 0 
\]
with each $c_i$ belonging to $\mathrm{Cof}(kG)$, then we say that $x$ admits a \textit{Benson cofibrant resolution} of length $n$, and we write $\mathrm{Cofdim}_{kG}(x)\leqslant n$.
The \textit{Benson cofibrant dimension} $\mathrm{Cofdim}_{kG}(x)$ is defined to be the length of a shortest such resolution, provided one exists.
If $x$ does not admit a finite Benson cofibrant resolution, then
we write $\mathrm{Cofdim}_{kG}(x)=\infty$.

Similarly, replacing the class $\mathrm{Cof}(kG)$ by $\mathrm{PCof}(kG)$,
we define \textit{proper cofibrant resolutions} and the \textit{proper cofibrant dimension} $\mathrm{PCofdim}_{kG}(x)$.
\end{definition}

\begin{remark}
Since both $(\mathrm{Cof}(kG),\mathrm{Cof(kG)}^{\perp})$ and $(\mathrm{PCof}(kG),\mathrm{PCof(kG)}^{\perp})$ are complete cotorsion pairs in the category of all $kG$--modules, resolutions by Benson cofibrants and proper cofibrants always exist, and they can even be constructed  unique up to homotopy; see for instance \cite[Section 8.1]{EJ}.
\end{remark}

\begin{recollection}
\label{rec:classes}
As already mentioned, there are inclusions
\begin{equation}
\label{eq:containments}
\mathrm{Proj}(kG)\subseteq\mathrm{PCof}(kG)\subseteq \mathrm{Cof}(kG)\subseteq \mathrm{GProj}(kG).
\end{equation}
These inclusions imply that, for every $kG$--module $x$, one has
\begin{equation}
\label{eq:inequalities}
 \mathrm{Gpd}_{kG}(x) \leqslant \mathrm{Cofdim}_{kG}(x)\leqslant \mathrm{PCofdim}_{kG}(x)  \leqslant\mathrm{pd}_{kG}(x)
\end{equation}
where by $\mathrm{Gpd}_{kG}(x)$ and $\mathrm{pd}_{kG}(x)$ we denote the Gorenstein projective dimension and the projective dimension of $x$, respectively. 
\end{recollection}
It is well known that $\mathrm{Gpd}$ refines $\mathrm{pd}$; 
that is, whenever $\mathrm{pd}_{kG}(x)<\infty$, one has
$\mathrm{Gpd}_{kG}(x)=\mathrm{pd}_{kG}(x)$. 
We show in \cref{cor:refinement} that $\mathrm{Cofdim}$ similarly refines $\mathrm{PCofdim}$. 
With a similar argument one can prove that $\mathrm{Gpd}$ refines $\mathrm{Cofdim}$ by combining
\cite[Lemma~4.5.2]{benson97} with \cite[Proposition~2.2]{ET_cof}. Consequently, each of the homological dimensions appearing in \eqref{eq:inequalities} refines those to its right.

\begin{remark}
\label{rem:finiteness}
The classes $\mathrm{Cof}(kG)$ and $\mathrm{PCof}(kG)$ are closed under coproducts, extensions, kernels of epimorphisms and contain all the projective $kG$--modules. It is a formal consequence of these properties that the finiteness of the homological dimensions $\mathrm{Cofdim}_{kG}$ and $\mathrm{PCofdim}_{kG}$ satisfies the 2-out-of-3 property for short exact sequences and a standard comparison formula for short exact sequences. 
For a proof, the reader may consult, for instance, \cite[Section~2]{DE}; the arguments given there apply verbatim in the present setting.
\end{remark}

\begin{proposition}
\label{prop:sex_cof}
Let $x$ be a $kG$--module with $\mathrm{PCofdim}_{kG}(x)=n>0$. Then there exists a short exact sequence of $kG$--modules $0\rightarrow z\rightarrow q\rightarrow x\rightarrow 0$ where $q\in\mathrm{PCof}(kG)$ and $\mathrm{pd}_{kG}(z)=n-1.$
\end{proposition}

\begin{proof}
The proof is akin to \cite[Theorem~2.10]{Holm}. Since $\mathrm{PCof}(kG)$ is projectively resolving, 
closed under direct sums and closed under direct summands, a standard argument implies that the $n$-th syzygy $y$ in a projective resolution of $x$ belongs in $\mathrm{PCof}(kG)$. By \cref{coro-wfin-complete} there exists an exact complex 
\[
0\rightarrow y\rightarrow q^0\rightarrow q^1\rightarrow \dots \rightarrow q^{n-1}\rightarrow x\rightarrow 0
\]
whose syzygies belong in $\mathrm{PCof}(kG)$ and the modules $q^0,q^1,\dots,q^{n-1}$ belong in $\mathrm{PCof}(kG)\cap\mathrm{PCof}(kG)^{\bot}=\mathrm{Proj}(kG)$. In addition, this complex remains exact after applying the functor $\Hom_{kG}(-,p)$, for any projective module $p$. Indeed,  this follows for instance since $\mathrm{PCof}(kG)\subseteq \mathrm{GProj}(kG)$. Now, as in \cite[Theorem~2.10]{Holm}, by comparing this complex with the projective resolution of $x$ that has $y$ as its $n$-th syzygy, a mapping cone argument shows the existence of a short exact sequence $0\rightarrow z\rightarrow q\rightarrow x\rightarrow 0$ where $q\in\mathrm{PCof}(kG)$ and $\mathrm{pd}_{kG}(z)=n-1.$
\end{proof}

\begin{corollary}
\label{cor:pcofs}
Let $x$ be a $kG$--module in $\mathrm{Cof}(kG)$ such that $\mathrm{PCofdim}_{kG}(x)<\infty$. Then $x\in\mathrm{PCof}(kG)$.
\end{corollary}

\begin{proof}
By \cref{prop:sex_cof} there exists a short exact sequence $0\rightarrow z\rightarrow q\rightarrow x\rightarrow 0$ with $q\in\mathrm{PCof}(kG)$ and $\mathrm{pd}_{kG}(z)<\infty.$ Since $x\in\mathrm{Cof}(kG)$ this sequence splits. The class $\mathrm{PCof}(kG)$ is closed under summands, thus  $x\in\mathrm{PCof}(kG).$
\end{proof}

\begin{corollary}
\label{cor:refinement}
Let $x$ be a $kG$--module with $\mathrm{PCofdim}_{kG}(x)<\infty$. Then the equality $\mathrm{Cofdim}_{kG}(x)=\mathrm{PCofdim}_{kG}(x)$ holds true.
\end{corollary}

\begin{proof}
Let $\mathrm{PCofdim}_{kG}(x)=n$, and consider a proper cofibrant resolution of $x$ of length $n$ 
\[
0\rightarrow q_n\rightarrow q_{n-1}\rightarrow \cdots\rightarrow q_0\rightarrow x\rightarrow 0.
\]
From \cref{eq:inequalities} we know that $\mathrm{Cofdim}_{kG}(x)\leqslant n$.   Assume that $\mathrm{Cofdim}_{kG}(x)=\lambda<n$. In this case, by the exact complex considered above it follows that  $c_{\lambda}\coloneqq\mathrm{coker}(q_{\lambda+1}\rightarrow q_\lambda)$ belongs in $\mathrm{Cof}(kG).$ But $c_\lambda$ is also of finite proper cofibrant dimension, and therefore belongs in $\mathrm{PCof}(kG)$ by \cref{cor:pcofs}. This implies that $\mathrm{PCofdim}(x)\leqslant\lambda$ which is a contradiction. Hence $\mathrm{Cofdim}_{kG}(x)=n$.
\end{proof}

\begin{definition}\label{def-Bensondim}
The \textit{Benson cofibrant cohomological dimension} of $G$ is defined as $\mathrm{BCcd}_kG\coloneqq  \mathrm{Cofdim}_{kG}(k)$. Similarly, the \textit{proper cofibrant cohomological dimension} of $G$ is defined as $\mathrm{PCcd_{k}G\coloneq}\mathrm{PCofdim}_{kG}(k)$.
\end{definition}

We are interested in characterizing the finiteness of the dimensions $\mathrm{BCcd}_kG$ and $\mathrm{PCcd}_{k}G$. 
In \cite{em-t}, the authors proved that, when $k$ has finite  global dimension, the Gorenstein cohomological dimension $\mathrm{Gcd}_kG$ is finite if and only if all $kG$--modules have finite Gorenstein projective dimension. 
Below, we establish analogous results for the Benson cofibrant and proper cofibrant dimensions. Our proofs follow some arguments from \cite{em-t}, which are themselves inspired by the seminal work of \cite{CP}.

\begin{proposition}
\label{prop:BC1}
For any commutative ring $k$ and group $G$ the following are equivalent:
\begin{itemize}
    \item[(i)] The trivial $kG$--module $k$ has finite Benson cofibrant dimension, that is, $\mathrm{BCcd}_{k}G<\infty$.
    \item[(ii)] There exists a $k$-split short exact sequence $0\rightarrow k\rightarrow w\rightarrow c\rightarrow 0$ of $kG$--modules where $\mathrm{pd}_{kG}(w)<\infty$ and $c\in\mathrm{Cof}(kG)$.
    \end{itemize}
\end{proposition}

\begin{proof}   
$(i)\Rightarrow(ii)$ From $\mathrm{Cofdim}_{kG}(k)<\infty$ we deduce that $\mathrm{Gpd}_{kG}(k)<\infty$  by \eqref{eq:inequalities}. Thus by \cite[Lemma~2.17]{CFH} there is a short exact sequence $0\rightarrow k\rightarrow w\rightarrow c\rightarrow 0$ where $\mathrm{pd}_{kG}(w)<\infty$ and $c\in \mathrm{GProj}(kG)$. In particular, $k$ and $w$ have finite Benson cofibrant dimension, thus the same holds for $c$. The $kG$--module $c$ is Gorenstein projective of finite Benson cofibrant dimension, hence $c\in\mathrm{Cof}(kG)$ by \cite[Proposition~2.2(i)]{ET_cof}. It follows by \cref{prop:cof_proj} that $c$ is $k$-projective and that the sequence $0\rightarrow k\rightarrow w\rightarrow c\rightarrow 0$ is $k$--split.

$(ii)\Rightarrow(i)$ In the given short exact sequence we have that $w$ and $c$ have finite Benson cofibrant dimension, thus the same holds for $k$.
\end{proof}

Recall that $\mathrm{silp}(kG)$ denotes the supremum of the injective lengths of projective $kG$--modules, and  that $\mathrm{spli}(kG)$ denotes the supremum of the projective lengths of injective $kG$--modules.

\begin{proposition}
\label{prop:BC2}
For any commutative ring $k$ and group $G$ the following are equivalent:
\begin{itemize}
    \item[(i)] All $kG$--modules have finite Benson cofibrant dimension.
    \item[(ii)] $\mathrm{spli}(kG)<\infty$ and $\mathrm{Cof}(kG)=\mathrm{GProj(kG)}$.
\end{itemize}
In case these conditions hold we have $\mathrm{gldim}(k)<\infty$.
\end{proposition}

\begin{proof}
$(i)\Rightarrow(ii)$ The assumption implies in particular that all $kG$--modules have finite Gorenstein projective dimension. This is equivalent, by \cite[Theorem~4.1]{CP}, to the finiteness of both invariants $\mathrm{silp}(kG)$ and $\mathrm{spli}(kG)$, which is actually equivalent to $\mathrm{spli}(kG)$ being  finite by \cite[Corollary~24]{DE}. In addition, in this case we have that $\mathrm{Cof}(kG)=\mathrm{GProj(kG)}$ by \cite[Proposition~2.2 (i)]{ET_cof}.

$(ii)\Rightarrow(i)$ Follows immediately by the aforementioned characterization of the finiteness of $\mathrm{spli}(kG)$.

Lastly, any $k$--module $x$ may be seen as a $kG$--module with trivial action. By assumption, $x$ admits a finite Benson cofibrant resolution as a $kG$--module. 
Since the modules in $\mathrm{Cof}(kG)$ are projective over $k$ by \cref{prop:cof_proj}, it follows that $x$ admits a finite projective resolution as a $k$--module. Hence $\mathrm{gldim}(k)<\infty.$
\end{proof}

\begin{theorem}
\label{thm:BCcd}
Let $G$ be a group and $k$ a commutative ring with $\mathrm{gldim}(k)<\infty$. Then the following are equivalent:
\begin{itemize}
    \item[(i)] The trivial $kG$--module $k$ has finite Benson cofibrant dimension, that is, $\mathrm{BCcd}_{k}G<\infty$.
    \item[(ii)] There exists a $k$-split short exact sequence $0\rightarrow k\rightarrow w\rightarrow c\rightarrow 0$ of $kG$--modules where $\mathrm{pd}_{kG}(w)<\infty$ and $c\in\mathrm{Cof}(kG)$.
    \item[(iii)] All $kG$--modules have finite Benson cofibrant dimension.
    \item[(iv)] $\mathrm{spli}(kG)<\infty$ and $\mathrm{Cof}(kG)=\mathrm{GProj(kG)}$.
\end{itemize}
\end{theorem}

\begin{proof}
Even without using $\mathrm{gldim}(k)<\infty$ the bi-implications $(i)\Leftrightarrow (ii)$ and $(iii)\Leftrightarrow (iv)$ hold by \cref{prop:BC1} and \cref{prop:BC2} respectively. In addition, $(iii)\Rightarrow (i)$ holds trivially. Now, assuming $\mathrm{gldim}(k)<\infty$, we prove that $(ii)\Rightarrow (iii)$. Let $x$ be any $kG$--module and assume $\mathrm{pd}_{kG}(w)=n.$ We claim that  
\[
\mathrm{Cofdim_{kG}}(x)\leqslant \mathrm{pd}_{kG}(w)+\mathrm{pd}_{k}(x).
\]
The proof is by induction on $\mathrm{pd}_{k}(x)$. Assume that $x$ is $k$--projective. We may apply $x\otimes -$ to the given $k$--split short exact sequence and obtain a short exact sequence $0\rightarrow x\rightarrow x\otimes w\rightarrow x\otimes c\rightarrow 0$ of $kG$--modules. Since $x$ is $k$--projective we deduce that $\mathrm{pd}_{kG}(x\otimes w)<\infty$ and that $x\otimes c\in\mathrm{Cof}(kG)$. Hence by \cref{rem:finiteness} we have
\[
\mathrm{Cofdim}_{kG}(x)\leqslant\mathrm{Cofdim}_{kG}(x\otimes w)=\mathrm{pd}_{kG}(x\otimes w)\leqslant\mathrm{pd}_{kG}(w).  
\]
This concludes the base case of the induction.
If we assume that $\mathrm{pd}_k(x)=m>0$ then we may consider a short exact sequence of $kG$--modules $0\rightarrow z'\rightarrow p \rightarrow x\rightarrow 0$ where $p$ is a projective $kG$--module. Then $\mathrm{pd}_k(z')=m-1$ and the induction hypothesis implies that $\mathrm{Cofdim}_{kG}(z')=\mathrm{pd}_{kG}(w)+m-1$. As a consequence, $\mathrm{Cofdim}_{kG}(x)=\mathrm{pd}_{kG}(w)+m$, which concludes the proof.
\end{proof}

Our next goal is to prove an analogous statement to  \cref{thm:BCcd} for the proper cofibrant dimension. The next result is analogous to \cite[Proposition~2.2(i)]{ET_cof}.

\begin{proposition}
\label{prop:PC1}
For any commutative ring $k$ and group $G$ the following are equivalent:
\begin{itemize}
    \item[(i)] The trivial $kG$--module $k$ has finite proper cofibrant dimension, that is, $\mathrm{PCcd}_{k}G<\infty$.
    \item[(ii)] There exists a $k$-split short exact sequence $0\rightarrow k\rightarrow w\rightarrow c\rightarrow 0$ of $kG$--modules where $\mathrm{pd}_{kG}(w)<\infty$ and $c\in\mathrm{PCof}(kG)$.
\end{itemize}
\end{proposition}

\begin{proof}
$(i)\Rightarrow(ii)$ By the inequalities displayed in \eqref{eq:inequalities} the finiteness of the invariant $\mathrm{PCofdim}_{kG}(k)$ implies the finiteness of $\mathrm{Cofdim}_{kG}(k)$. Thus by \cref{prop:BC1} there exists a short exact sequence $0\rightarrow k\rightarrow w\rightarrow c\rightarrow 0$ where $\mathrm{pd}_{kG}(w)<\infty$ and $c\in\mathrm{Cof}(kG)$. Since $k$ and $w$ have finite proper cofibrant dimension, the same holds for $c$ by \cref{rem:finiteness}. Thus from \cref{cor:pcofs} we deduce that $c\in\mathrm{PCof}(kG)$. 

$(ii)\Rightarrow(i)$ follows by the 2-out-3 property of the finiteness of $\mathrm{PCofdim}(kG)$.
\end{proof}

\begin{proposition}
\label{prop:PC2}
For any commutative ring $k$ and group $G$ the following are equivalent:
\begin{itemize}
    \item[(i)] All $kG$--modules have finite proper cofibrant dimension.
    \item[(ii)] $\mathrm{spli}(kG)<\infty$ and $\mathrm{PCof}(kG)=\mathrm{Cof}(kG)=\mathrm{GProj}(kG)$.
\end{itemize}
In case these conditions hold we have $\mathrm{gldim}(k)<\infty$.
\end{proposition}

\begin{proof}
Follows from \cref{prop:BC2} together with \cref{cor:pcofs}.
\end{proof}

\begin{theorem}
\label{thm:PCcd}
Let $G$ be a group and $k$ a commutative ring with $\mathrm{gldim}(k)<\infty$. Then the following are equivalent:
\begin{itemize}
\item[(i)] The trivial $kG$--module $k$ has finite proper cofibrant dimension, that is, $\mathrm{PCcd}_{k}G<\infty$.
\item[(ii)] There exists a $k$-split short exact sequence $0\rightarrow k\rightarrow w\rightarrow c\rightarrow 0$ of $kG$--modules where $\mathrm{pd}_{kG}(w)<\infty$ and $c\in\mathrm{PCof}(kG)$.
\item[(iii)] All $kG$--modules have finite proper cofibrant dimension.
\item[(iv)] $\mathrm{spli}(kG)<\infty$ and $\mathrm{PCof}(kG)=\mathrm{Cof}(kG)=\mathrm{GProj}(kG)$.
\end{itemize}
\end{theorem}

\begin{proof}
Similarly to the proof of \cref{thm:BCcd}, without using $\mathrm{gldim}(k)<\infty$ we have that $(i)\Leftrightarrow (ii)$ and $(iii)\Leftrightarrow (iv)$ hold by \cref{prop:PC1,prop:PC2}  respectively. In addition, $(iii)\Rightarrow (i)$ holds trivially. Now, assuming $\mathrm{gldim}(k)<\infty$, we prove that $(ii)\Rightarrow (iii)$. Let $x$ be any $kG$--module and assume $\mathrm{pd}_{kG}(w)=n.$ We claim that 
\[
\mathrm{PCofdim_{kG}}(x)\leqslant \mathrm{pd}_{kG}(w)+\mathrm{pd}_{k}(x)
\]
and our proof will be by induction on $\mathrm{pd}_{k}(x)$. Assume that $x$ is $k$--projective. We may apply $x\otimes -$ to the given $k$--split short exact sequence and obtain a short exact sequence $0\rightarrow x\rightarrow x\otimes w\rightarrow x\otimes c\rightarrow 0$ of $kG$--modules. From \cref{lem:tensors} we deduce that $x\otimes c\in\mathrm{PCof}(kG)$. 
Now, by \cref{rem:finiteness},
\[
\mathrm{PCofdim}_{kG}(x)\leqslant\mathrm{PCofdim}_{kG}(x\otimes w)=\mathrm{pd}_{kG}(x\otimes w)\leqslant\mathrm{pd}_{kG}(w),
\]
which concludes the proof of the base case of the induction.

If we assume that $\mathrm{pd}_k(x)=m>0$ then we may consider a short exact sequence of $kG$--modules $0\rightarrow z'\rightarrow p \rightarrow x\rightarrow 0$ where $p$ is a projective $kG$--module. Then $\mathrm{pd}_k(z')=m-1$ and the induction hypothesis implies that $\mathrm{PCofdim}_{kG}(z')=\mathrm{pd}_{kG}(w)+m-1$. As a consequence, $\mathrm{PCofdim}_{kG}(x)=\mathrm{pd}_{kG}(w)+m$, which concludes the proof.
\end{proof}

\begin{remark}
The existence of $k$--split monomorphisms $0\rightarrow k\rightarrow w$ where $w$ is a module of finite projective dimension, as found for instance in the context of  \cref{prop:BC1,prop:BC2}, is important for proving further homological properties. Such modules $w$ are called characteristic modules, see \cite{tal-char,em-tal-char}. For instance, the characteristic modules appearing in  \cref{prop:BC1,prop:BC2} can be used in order to obtain the following results, whose proofs are very similar to the ones found in \cite[Propositions~2.1 and 2.4]{em-t}. We leave the details to the interested reader.
\end{remark}

\begin{corollary}
\label{cor:integral}
For any group $G$ and commutative ring $k$ the inequality $\mathrm{BCcd}_kG\leqslant\mathrm{BCcd}_{\mathbb{Z}}G$ 
holds true.
\end{corollary}

\begin{corollary}
\label{cor:subgroups}
For any group $G$, commutative ring $k$ and subgroup $H\leqslant  G$ the inequalities $\mathrm{BCcd}_kH\leqslant \mathrm{BCcd}_kG$ and $\mathrm{PCcd}_kH\leqslant \mathrm{PCcd}_kG$ hold true.
\end{corollary}

\bibliographystyle{alpha}
\bibliography{bibfile}

\end{document}